\documentclass[11pt,a4paper]{article}
\usepackage{amsmath,amssymb,amsthm,amscd,mathrsfs}
\usepackage{indentfirst}
\usepackage[top=25mm, bottom=30mm, left=30mm, right=30mm]{geometry}
\allowdisplaybreaks[4]
\newtheorem{Def}{\bf Definition}[section]
\newtheorem{Thm}[Def]{\bf Theorem}
\newtheorem{Lem}[Def]{\bf Lemma}

\newtheorem{Pro}[Def]{\bf Proposition}
\newtheorem{Rem}[Def]{\bf Remark}

\newtheorem{step}{\bf Step}

\newtheorem*{claim}{\bf Claim}

\newtheorem*{conv}{\bf Convention}
\newtheorem{ThmA}{\bf Theorem}
\newtheorem{CorA}[ThmA]{\bf Corollary}

\newcommand{\R}{\mathbb{R}}
\newcommand{\C}{\mathbb{C}}

\newcommand{\F}{\mathbb{F}}

\newcommand{\N}{\mathbb{N}}
\newcommand{\B}{\mathbb{B}}
\newcommand{\K}{\mathbb{K}}
\newcommand{\M}{\mathbb{M}}
\newcommand{\G}{\mathbb{G}}
\newcommand{\Ad}{\operatorname{Ad}}
\newcommand{\id}{\text{\rm id}}
\newcommand{\Aut}{\operatorname{Aut}}
\newcommand{\Tr}{\mathord{\text{\rm Tr}}}
\newcommand{\ovt}{\mathbin{\overline{\otimes}}}

\newcommand{\otm}{\otimes_{\rm min}}
\newcommand{\ota}{\otimes_{\rm alg}}

\newcommand{\op}{\circ}
\title{\bf \Large{Cartan subalgebras of tensor products of \\free quantum group factors with arbitrary factors}}
\author{Yusuke Isono\thanks{Research Institute for Mathematical Sciences, Kyoto University, 606-8502, Kyoto, Japan \protect \\  E-mail: \texttt{isono@kurims.kyoto-u.ac.jp}}}
\date{}
\begin{document}
\maketitle

\begin{abstract}
	Let $\G$ be a free (unitary or orthogonal) quantum group. We prove that for any non-amenable subfactor $N\subset L^\infty(\G)$, which is an image of a faithful normal conditional expectation, and for any $\sigma$-finite factor $B$, the tensor product $N \ovt B$ has no Cartan subalgebras. 
This generalizes our previous work that provides the same result when $B$ is finite. 
In the proof, we establish Ozawa--Popa and Popa--Vaes's weakly compact action on the continuous core of $L^\infty(\G)\ovt B$ as the one \textit{relative to $B$}, by using an operator valued weight to $B$ and the central weak amenability of $\widehat{\G}$. 
\end{abstract}

\section{\bf Introduction}

	Let $M$ be a von Neumann algebra. A \textit{Cartan subalgebra}  $A \subset M$ is an abelian von Neumann subalgebra, which is an image of a faithful normal conditional expectation, such that (i) $A$ is maximal abelian and (ii) the normalizer $\mathcal{N}_M(A)$ generates $M$ as a von Neumann algebra \cite{FM75}. Here $\mathcal{N}_M(A)$ is given by $\{ u \in \mathcal{U}(M) \mid uAu^*=A \}$. 

	The group measure space construction of Murray--von Neumann gives a typical example of a Cartan subalgebra. Indeed, the canonical subalgebra $L^\infty(X,\mu) \subset L^\infty (X,\mu) \rtimes \Gamma$ is Cartan whenever the given action $\Gamma\curvearrowright (X,\mu)$ is free. More generally, one can associate any (not necessarily free) group action with a Cartan subalgebra by its orbit equivalence relation. Conversely when $M$ has separable predual, any Cartan subalgebra $A \subset M$ is realized by an orbit equivalence relation (with a cocycle), and hence by a group action. Thus the notion of Cartan subalgebras is closely related to group actions. 
In particular if $M$ has no Cartan subalgebras, then it can not be constructed by any group actions. It was an open problem to find such a von Neumann algebra.

	The first result for this direction was given by Connes. He constructed a $\rm II_1$ factor which is not isomorphic to its opposite algebra, so it is particularly not isomorphic to any group action (without cocycle) von Neumann algebra \cite{Co74}. 
Voiculescu then provided a complete solution to this problem, by proving free group factors $L\F_n$ $(n\geq 2)$ have no Cartan subalgebras \cite{Vo95}. 
He used his celebrated \textit{free entropy} techniques, and it was later developed to give other examples \cite{Sh00,Ju05}.

	After these pioneering works, Ozawa and Popa introduced a completely new framework to study this subject. Among other things, they proved that free group factors are \textit{strongly solid} \cite{OP07}, that is, for any diffuse amenable subalgebra $A \subset L\F_n$, the von Neumann algebra generated by the normalizer $\mathcal{N}_{L\F_n}(A)$ remains amenable. Since $L\F_n$ itself is non-amenable, this immediately yields that $L\F_n$ has no Cartan subalgebras. Note that strong solidity is stable under taking subalgebras and hence any non-amenable subfactor of $L\F_n$ also has no Cartan subalgebras. 

	The proof of Ozawa and Popa consist of two independent steps. First, by using weak amenability of $\F_n$, they observed that the normalizer group acts \textit{weakly compactly} on a given amenable subalgebra. Second, combining this weakly compact action with \textit{Popa's deformation and intertwining techniques}  \cite{Po01,Po03}, they constructed a state which is central with respect to the normalizer group. Thus they obtained that the normalizer group generates an amenable von Neumann algebra. 
Since these techniques are applied to any finite crossed product $B \rtimes \F_n$ with the W$^*$CMAP (see Subsection \ref{WCMAP}), 
they also proved that for any finite factor $B$ with the W$^*$CMAP, the tensor product $L\F_n \ovt B$ has no Cartan subalgebras. 

	To remove the W$^*$CMAP assumption on $B\rtimes \F_n$, Popa and Vaes introduced a notion of \textit{relative weakly compact action} \cite{PV11}. This is an appropriate ``relativization'' of the first step above in the view of the relative tensor product $L^2(B\rtimes\F_n)\otimes_B L^2(B\rtimes \F_n)$. In particular this \textit{only} requires the weak amenability of $\F_n$. Thus by modifying the proof in the second step above, they obtained, among other things, the tensor product $L\F_n \ovt B$ has no Cartan subalgebras for any finite factor $B$. 

	The aim of the present paper is to develop these techniques to study type III von Neumann algebras. More specifically we replace the free group factor $L\F_n$ with the \textit{free quantum group factor}, which is a type III factor in most cases. 
This has already been studied by ourselves \cite{Is12,Is13} when $B$ is finite. In the general case however, namely, when $B$ is a type III factor, we could not provide a satisfactory answer to this problem, and this will be discussed in this article.

	We note that the first solution to the Cartan subalgebra problem for type III factors in our framework was obtained by Houdayer and Ricard \cite{HR10}. They followed the proof of \cite{OP07} by exploiting techniques in \cite{CH08}, that is, the use of Popa's deformation and intertwining techniques together with the \textit{continuous core decomposition}. While Houdayer and Ricard followed the idea of \cite{OP07}, our approach in \cite{Is12,Is13} was based on \cite{PV12}. In particular, in the second step above, we made use of \textit{Ozawa's condition (AO)} \cite{Oz03} (or \textit{bi-exactness}, see Subsection \ref{Free quantum groups and bi-exactness}) at the level of the continuous core. 
In this article, we stand again on the use of bi-exactness, and we will further develop  techniques of \cite{Is13}. 
See \cite{BHR12} for other examples of type III factors with no Cartan subalgebras, and \cite{CS11,CSU11} for other works on Cartan subalgebras of bi-exact group von Neumann algebras. 

	The following theorem is the main observation of this article. This should be regarded as a generalization of \cite[Theorem B]{Is13}, and this allows us to obtain a satisfactory answer to the Cartan problem in the type III setting. 
See Section \ref{Preliminaries} for items in this theorem.

\begin{ThmA}\label{thmA}
	Let $\G$ be a compact quantum group with the Haar state $h$, and $B$ a type $\rm III_1$ factor with a faithful normal state $\varphi_B$. Put $M:=L^\infty (\G)\ovt B$ and $\varphi:=h\otimes \varphi_B$. 
Let $C_{\varphi_B}(B)$ and $C_\varphi(M)$ be continuous cores of $B$ and $M$ with respect to $\varphi_B$ and $\varphi$, and regard $C_{\varphi_B}(B)\subset C_\varphi(M)$. Let $\Tr$ be a semifinite trace on $C_\varphi(M)$ with $\Tr|_{C_{\varphi_B}(B)}$ semifinite, and $p\in C_\varphi(M)$ a projection with $\Tr(p)<\infty$. 

	Assume that $\widehat{\G}$ is bi-exact and centrally weakly amenable with Cowling--Haagerup constant 1. Then for any amenable von Neumann subalgebra $A\subset pC_\varphi(M)p$, we have either one of the following conditions.
	\begin{itemize}
		\item[$\rm (i)$] We have $A\preceq_{C_\varphi(M)} C_{\varphi_B}(B)$.
		\item[$\rm (ii)$] The von Neumann algebra $\mathcal{N}_{pC_\varphi(M)p}(A)''$ is amenable relative to $C_{\varphi_B}(B)$.
	\end{itemize}
\end{ThmA}

	As a consequence of the main theorem, we obtain the following corollary. This is the desired one since our main example, free quantum groups, satisfy assumptions in this corollary. See \cite[Theorem C]{Is13} for other examples of quantum groups satisfying these assumptions. 
Below we say that an inclusion of von Neumann algebras $A\subset M$ is \textit{with expectation} if there is a faithful normal conditional expectation.

\begin{CorA}\label{corB}
	Let $\G$ be a compact quantum group as in Theorem \ref{thmA}. Then for any non-amenable subfactor $N \subset L^\infty(\G)$ with expectation and any $\sigma$-finite factor $B$, the tensor product $N \ovt B$ has no Cartan subalgebras.
\end{CorA}

	For the proof of Theorem \ref{thmA},	we will establish a weakly compact action on the continuous core of $L^\infty(\G)\ovt B$ as the one \textit{relative to $B$}. 
The central weak amenability of $\widehat{\G}$ is used to find approximation maps on the continuous core, which are relative to $B\rtimes \R$. Then combined with the amenability of $\R$, we construct appropriate approximation maps on the core relative to $B$. In this process, since $B$ is not with expectation in the core, we use operator valued weighs instead. This is our strategy for the first step.

	For the second step, although we go along a very similar line to \cite{Is13}, we need a rather different (and general) approach to the proof. We note that this is why we assume only bi-exactness of $\widehat{\G}$, and do not need the notion of \textit{condition (AOC)$^+$} as in \cite{Is12,Is13}.

\bigskip

	This paper organizes as follows. 
In Section \ref{Preliminaries}, we recall fundamental facts for our paper, such as Tomita--Takesaki theory, free quantum groups, bi-exactness, weak amenability, and Popa's intertwining techniques.

	In Section \ref{Weakly compact actions}, we study a generalization of the relative weakly compact action on the continuous core, by constructing appropriate approximation maps on the core. The main tools for this construction are: operator valued weights; central weak amenability; and weak containment, together with the amenability of $\R$. This is the most technical part of this paper.

	In Section \ref{Proof of Theorem A free quantum group factors}, we prove the main theorem. We follow the proof of \cite{PV12,Is13}, using the weakly compact action given in Section \ref{Weakly compact actions}.

\bigskip

\noindent 
{\bf Acknowledgement.} The author would like to thank Yuki Arano, Kei Hasegawa, and Reiji Tomatsu  for useful comments on the relative amenability, and Narutaka Ozawa and Stefaan Vaes for fruitful conversations. 
This research was carried out while he was visiting the University of California Los Angeles. He gratefully acknowledges the kind hospitality of them. 
He was supported by JSPS, Research Fellow of the Japan Society for the Promotion of Science.


\section{\bf Preliminaries}\label{Preliminaries}

\subsection{\bf Tomita--Takesaki theory and operator valued weights}\label{Tomita Takesaki theory and operator valued weights}

	We first recall some notions in Tomita--Takesaki theory. We refer the reader to \cite{Ta01} for this theory, and to \cite{Ha77a,Ha77b} and \cite[Chapter IX.$\S$4]{Ta01} for operator valued weights.

	Let $M$ be a von Neumann algebra and $\varphi$ a faithful normal semifinite weight on $M$. Put $\mathfrak{n}_\varphi:=\{ x\in M\mid \varphi(x^*x)<\infty\}$ and denote by $\Lambda_\varphi \colon \mathfrak{n}_\varphi \to L^2(M,\varphi)$ the canonical embedding. We denote the \textit{modular operator, modular conjugation}, and \textit{modular action} for $M \subset \B(L^2(M,\varphi))$ by $\Delta_\varphi$, $J_\varphi$ and $\sigma^{\varphi}$ respectively. 
The Hilbert space $L^2(M,\varphi)$ with $J_\varphi$ and with its \textit{positive cone $\mathcal{P}_\varphi$} is called the \textit{standard representation for $M$} \cite[Chapter IX.$\S$1]{Ta01}, which does not depends on the choice of $\varphi$. Any state on $M$ is represented by a vector state, from which the vector is uniquely chosen from $\mathcal{P}_\varphi$. Any element $\alpha\in \Aut(M)$ is written as $\alpha=\Ad u $ by a unique $u \in \B(L^2(M,\varphi))$ which preserves the standard representation structure. 
The crossed product $M \rtimes_{\sigma^\varphi} \R$ by the modular action is called the \textit{continuous core} \cite[Chapter XII.$\S$1]{Ta01} and is written as $C_\varphi(M)$, which is equipped with the dual weight $\widehat{\varphi}$ and the canonical trace $\Tr_\varphi:=\widehat{\varphi}(h_\varphi^{-1} \, \cdot \,)$, where $h_\varphi$ is a self-adjoint positive closed operator affiliated with $L\R$. 
For any other faithful normal semifinite weigh $\psi$, there is a family of unitaries  $([D\varphi, D\psi]_t)_{t\in \R}$ in $M$ called the \textit{Connes cocycle}  \cite[Definition VIII.3.4]{Ta01}. This gives a cocycle conjugate for modular actions of $\varphi$ and $\psi$, and hence there is a $\ast$-isomorphism 
	$$\Pi_{\psi,\varphi}\colon C_{\varphi}(M)\to C_{\psi}(M), \quad \Pi_{\psi,\varphi}(x)=x \quad (x\in M), \quad \Pi_{\psi,\varphi}(\lambda_t^\varphi)=[D\psi, D\varphi]_t^*\lambda_t^\psi  \quad  (t\in \R).$$
It holds that $\Pi_{\psi,\varphi}\circ \Pi_{\varphi, \omega}=\Pi_{\psi,\omega}$ for any other $\omega$ on $M$, and $\Pi_{\psi\circ E_M ,\varphi\circ E_M} |_{C_{\varphi}(M)} = \Pi_{\psi,\varphi}$ for any $M\subset N$ with expectation $E_M$. 
It preserves traces $\Tr_{\psi}\circ \Pi_{\psi,\varphi} = \Tr_{\varphi}$ \cite[Theorem XII.6.10(iv)]{Ta01}. So the pair $(C_\varphi(M),\Tr_\varphi)$ does not depend on the choice of $\varphi$, and we call $\Tr_\varphi$ the canonical trace. 
A von Neumann algebra is said to be a \textit{type III$_1$ factor} if its continuous core is a II$_\infty$ factor.

	Let $B\subset M$ be any inclusion of von Neumann algebras. We write as $\widehat{M}^+$ the \textit{extended positive cone} of $M$. For any \textit{operator valued weight} $T \colon \widehat{M}^+\rightarrow \widehat{B}^+$, we use the notation 
\begin{align*}
	\mathfrak{n}_{T}&:= \left\{x\in M\mid \| T(x^*x)\|_\infty<+\infty \right\},\\
	\mathfrak{m}_{T}&:=(\mathfrak{n}_{T})^*\mathfrak{n}_{T}= \left\{\sum_{i=1}^nx_i^*y_i \mid n \geq 1, x_i, y_i\in \mathfrak{n}_{T} \text{ for all } 1 \leq i \leq n \right\}.
\end{align*}
Then $T$ has a unique extension $T\colon \mathfrak{m}_{T}\rightarrow B$ as  a $B$-bimodule linear map. In this paper, all the operator valued weights that we consider are assumed to be \textit{faithful, normal} and \textit{semifinite}. Note that since the operator valued weight is nothing but a weight when $B=\C$, we may also extend a faithful normal semifinite weight $\varphi$ on $\mathfrak{m}_\varphi$.

	For any inclusion $B\subset M$ of von Neumann algebras with faithful normal weights $\varphi_B$ and $\varphi_M$ on $B$ and $M$ respectively, the modular actions of them satisfy $\sigma^{\varphi_M}|_B=\sigma^{\varphi_B}$ if and only if there is an operator valued weight $E_B$ from $M$ to $B$ which satisfies $\varphi_B\circ E_B=\varphi_M$, and $E_B$ determines uniquely by this equality \cite[Theorem IX.4.18]{Ta01}. We call $E_B$ the \textit{operator valued weight from $(M,\varphi_M)$ to $(B,\varphi_B)$}.  
In this case, the cores has an inclusion $C_{\varphi_B}(B) \subset C_{\varphi_M}(M)$ since $\sigma^{\varphi_M} |_B=\sigma^{\varphi_B}$.
When $\varphi_M|_B=\varphi_B$, $E_B$ is a faithful normal conditional expectation \cite[Theorem IX.4.2]{Ta01}.

	Let $M$ be a von Neumann algebra and $\varphi$ a faithful normal semifinite weight on $M$. Put $L^2(M):=L^2(M,\varphi)$ and let $\alpha$ be an action of $\R$ on $M$. In this article, as a representation of $M\rtimes_\alpha \R$, we use that for any $\xi \in L^2(\R)\otimes L^2(M)\simeq L^2(\R ,M)$ and $s,t \in \R$,
\begin{align*}
	& M \ni x \mapsto \pi_{\alpha}(x) ; \ (\pi_{\alpha}(x)\xi)(s):=\alpha_{-s}(x) \xi (s);\\
	& L\R \ni \lambda_t \mapsto 1_M \otimes \lambda_t ; \ ((1\otimes \lambda_t)\xi)(s) := \xi(s-t).
\end{align*}
Let $C_c(\R , M)$ be the set of all $\ast$-strongly continuous functions from $\R$ to $M$ with compact supports. Then there is an embedding
	$$\widehat{\pi}_\alpha\colon C_c(\R , M) \ni f \mapsto  \int_\R (1\otimes \lambda_t)\pi_{\alpha}(f(t)) dt \in M\rtimes_\alpha \R, $$
where the integral here should be understood as the map $T\in \B(L^2(\R)\otimes L^2(M))$ given by $\langle T\xi , \eta\rangle = \int_\R \langle (1\otimes \lambda_t) \pi_\alpha(f(t)) \xi, \eta \rangle  dt$ for all $\xi,\eta \in L^2(\R)\otimes L^2( M)$. We note that by $(f*g)(t):=\int_\R\alpha_s(f(t+s))g(-s) ds$ and $f^\sharp(t):=\alpha_{t}^{-1}(f(-t)^*)$ for $f,g\in C_c(\R, M)$ and $t\in \R$, $C_c(\R, M)$ is a $\ast$-algebra, so that $\widehat{\pi}_\alpha$ is a $\ast$-homomorphism. 
For $f\in C_c(\R,M)$ and $x\in M$, we denote by $(f\cdot x)(t):=f(t)x$ for $t\in G$. Let $C_c(\R ,M) \mathfrak{n}_\varphi \subset C_c(\R ,M) $ be the set of linear spans of $f\cdot x$ for $f\in C_c(\R,M)$ and $x\in \mathfrak{n}_\varphi$. 
With these notation, the dual weight satisfies 
	$$\widehat{\varphi}( \widehat{\pi}_\alpha(g)^* \widehat{\pi}_\alpha(f)) = \varphi((g^\sharp \ast f) (0)) = \int_\R \varphi(g(t)^* f(t)) dt \quad \text{ for any } f,g \in C_c(\R ,M) \mathfrak{n}_\varphi$$ 
\cite[Theorem X.1.17]{Ta01}. The modular objects of $\widehat{\varphi}$ are given by 
\begin{align*}
	&\sigma_t^{\widehat{\varphi}}|_M=\sigma_t^\varphi \quad \text{and} \quad \sigma_t^{\widehat{\varphi}}(\lambda_s)=\lambda_s [D(\varphi\circ \alpha_s), D\varphi]_t , \quad \text{for } s,t \in \R;\\
	& (J_{\widehat{\varphi}}\xi ) (t) = u^*(t) J_\varphi\xi(-t), \quad \text{for } t \in \R \text{ and } \xi \in L^2(\R , L^2(M)),
\end{align*}
where $u(t)$ is the unitary such that $\alpha_t = \Ad u(t)$ which preserves the standard structure of $L^2(M,\varphi)$. 
In particular $\sigma^{\widehat{\varphi}}$ globally preserves $M$ and so there is a canonical operator valued weight $E_M$ from $(M\rtimes_{\alpha} \R, \widehat{\varphi})$ to $(M,\varphi)$. By the equality $\varphi\circ E_M = \widehat{\varphi}$, it holds that for any $f ,g\in C_c(\R ,M)$,
	$$E_M( \widehat{\pi}_\alpha(g)^* \widehat{\pi}_\alpha(f)) = (g^\sharp \ast f) (0) = \int_\R g(t)^* f(t) dt .$$ 

Here we prove a few lemmas.

\begin{Lem}\label{operator valued weight}
	Let $(N,\varphi_N)$ and $(B,\varphi_B)$ be von Neumann algebras with faithful normal semifinite weights with $\varphi_N(1)=1$. Let $\alpha^B$ be an action of $\R$ on $B$, and put $M:=N\ovt B$, $\varphi:=\varphi_N\otimes \varphi_B$, $\alpha:=\sigma^{\varphi_N}\otimes \alpha^B$. Let $E_M$, $E_B$, $E_{B\rtimes \R}$ be the canonical operator valued weights from $(M\rtimes_\alpha \R, \widehat{\varphi})$ to $(M,\varphi)$, from $(M\rtimes_\alpha\R, \widehat{\varphi})$ to $(B,\varphi_B)$, and from $(M\rtimes_\alpha \R, \widehat{\varphi})$ to $(B\rtimes_{\alpha^B}\R, \widehat{\varphi}_B)$ respectively. 
Then we have $E_{B\rtimes \R} \circ E_M = E_B$.
\end{Lem}
\begin{proof}
	Let $P_N$ be the one dimensional projection from $L^2(N,\varphi_N)$ onto $\C \Lambda_{\varphi_N}(1_N)$ and observe that the compression map by $P_N \otimes 1_B \otimes 1_{L^2(\R)}$ on $N \ovt B \ovt \B(L^2(\R))$ gives a normal conditional expectation $E\colon M\rtimes_\alpha\R \to B\rtimes_{\alpha^B}\R$ satisfying $E((x \otimes b)\lambda_t) = \varphi_N(x)b\lambda_t$ for $x\in N$, $b\in B$, and $t\in \R$. It is faithful since so is on $N \ovt B \ovt \B(L^2(\R))$. A simple computation shows that $E=E_{B\rtimes \R}$ and $E_{B\rtimes \R}((x \otimes b)\lambda_t) = \varphi_N(x)b\lambda_t$ for $x\in N$, $b\in B$, and $t\in \R$. In particular $E_{B\rtimes \R}|_M$ is the canonical conditional expectation $E_B^M$ from $(M,\varphi)$ to $(B,\varphi_B)$. 
Then by definition, $\varphi_B\circ E_{B}^M \circ E_M= \varphi \circ E_M=\widehat{\varphi}$, and hence $E_{B}^M \circ E_M = E_B$. Since $E_{B}^M \circ E_M= E_{B\rtimes \R}\circ E_M$, we obtain the conclusion.
\end{proof}

	We next recall the following well known fact. We include a proof for reader's convenience.

\begin{Lem}\label{center of core of III1}
	Let $M$ be a type $\rm III_1$ factor and $N$ a von Neumann algebra. Then the center of the continuous core of $M\ovt N$ coincides with the center of $N$.
\end{Lem}
\begin{proof}
	Since $M$ is a type III$_1$ factor, there is a faithful normal semifinite weight $\varphi_M$ on $M$ such that $(M_{\varphi_M})'\cap M=\C$ \cite[Theorem XII.1.7]{Ta01}, where $M_{\varphi_M}$ is the fixed point algebra of the modular action of $\varphi_M$. Let $\varphi_N$ be a faithful normal semifinite weight on $N$ and put $\varphi:=\varphi_M \otimes \varphi_N$. Observe that the center of $C_\varphi(M\ovt N)$ is contained in 
	$$(M_{\varphi_M} \otimes \C1_{L^2(N)\otimes L^2(\R)})' \cap M\ovt N \ovt \B(L^2(\R)) = \C 1_{L^2(M,\varphi_M)}\ovt N\ovt \B (L^2(\R)).$$
On the other hand, since $\mathcal{Z}(C_\varphi(M\ovt N))$ commutes with $L\R$, it is contained in $ (M\ovt N)_{\varphi} \ovt L\R $ (e.g.\ \cite[Proposition 2.4]{HR10}). Hence 
	$$\mathcal{Z}(C_\varphi(M\ovt N)) \subset \C\ovt N\ovt \B (L^2(\R)) \cap (M\ovt N)_{\varphi} \ovt L\R = \C  \ovt N_{\varphi_N} \ovt L\R.$$
Finally since $\mathcal{Z}(C_\varphi(M\ovt N))$ commutes with $M$, and $N_{\varphi_N}$ commutes with $M$ and $L\R$, (up to exchanging positions of $M$ and $N$,) we have 
	$$\mathcal{Z}(C_\varphi(N\ovt M)) \subset M' \cap   N_{\varphi_N} \ovt \C\ovt L\R =    N_{\varphi_N} \ovt (M'\cap \C\ovt L\R ) = N_{\varphi_N} \ovt \C1,$$
where we used $M'\cap \C\ovt L\R \subset \mathcal{Z}(C_{\varphi_M}(M)) = \C$. Since $N'\cap N_{\varphi_N} = \mathcal{Z}(N)$, we conclude that $\mathcal{Z}(C_\varphi(M\ovt N))=\mathcal{Z}(N)$. Since all continuous cores are isomorphic with each other preserving the position of $M\ovt N$, for any other faithful normal semifinite weight $\psi$, we obtain $\mathcal{Z}(C_\psi(M\ovt N))=\mathcal{Z}(N)$.
\end{proof}

\subsection{\bf Relative tensor products, basic constructions and weak containments}\label{Relative tensor products, basic constructions and weak containments}

	Let $M$ and $N$ be von Neumann algebras and $H$ a Hilbert space. 
Throughout this paper, we denote \textit{opposite} objects with circle (e.g.\  $N^\op:=N^{\rm op}$, $x^\op:=x^{\rm op}\in N^\op$, $(xy)^\op=y^\op x^\op$ for $x,y\in N$). 
We say that $H$ is a \textit{left $M$-module} (resp.\ a \textit{right $N$-module}) if there is a normal unital injective $\ast$-homomorphism $\pi_H \colon M \to \B(H)$ (resp.\ $\theta_H \colon N^\op  \to \B(H)$). We say $H$ is an \textit{$M$-$N$-bimodule} if $H$ is a left $M$-module and a right $N$-module with commuting ranges. 
The \textit{standard bimodule} of $M$ is a standard representation $L^2(M)$ as an $M$-bimodule, where the right action is given by $M^\op  \ni x^\op  \mapsto Jx^*J \in M' \subset \B(L^2(M))$.

	Let $N$ be a von Neumann algebra, $\varphi$ a faithful normal semifinite weight, and $H=H_N$ a right $N$-module with the right action $\theta$. A vector $\xi\in H$ is said to be \textit{left $\varphi$-bounded} if there is a constant $C>0$ such that $\|\theta(x^\op )\xi\| \leq C \|J_\varphi\Lambda_\varphi(x^*) \|$ for all $x\in \mathfrak{n}_\varphi^*$. 
We denote by $D(H, \varphi)$ all left $\varphi$-bounded vectors in $H$. It is known that the subspace $D(H, \varphi)\subset H$ is always dense \cite[Lemma IX.3.3(iii)]{Ta01}. For $\xi\in D(H,\varphi)$, define a bounded operator
$$L_\xi\colon L^2(N,\varphi)\to H; \ L_\xi J_\varphi\Lambda_\varphi(a^*)= \theta(a^\op ) \xi.$$ 
It is easy to verify that
\begin{itemize}
	\item $\theta(x^\op )L_\xi=L_\xi J_\varphi x^*J_\varphi$ \quad $(x\in N)$;
	\item $L_\xi L_\eta^*\in \theta(N^\op )'$ and $L_\eta^* L_\xi \in (J_\varphi NJ_\varphi )'=N$ \quad $(\xi,\eta\in D(H,\varphi))$;
	\item $xL_\xi y= L_{x \theta(\sigma_{i/2}^\varphi(y)^\op ) \xi }$ \quad $(x\in \theta(N^\op )', y\in N_a)$,
\end{itemize}
where $N_a\subset N$ is the subalgebra consisting of all \textit{analytic} elements with respect to $(\sigma_t^\varphi)$ (see \cite[Lemma IX.3.3(v)]{Ta01} for the third statement).
For a left $N$-module $K={}_NK$, the \textit{relative tensor product} $H\otimes_N K$ is defined as the Hilbert space obtained by separation and compression of $D(H,\varphi) \otimes_{\rm alg}K$ with an inner product $\langle \xi_1\otimes_N \eta_1, \xi_2\otimes_N \eta_2 \rangle:= \langle L_{\xi_2}^*L_{\xi_1} \eta_1, \eta_2 \rangle_K$. When $H={}_MH_N$ is an $M$-$N$-bimodule and $K={}_N K_A$ is an $N$-$A$-bimodule for von Neumann algebras $M$ and $A$, the Hilbert space $H\otimes_N K$ is an $M$-$A$-bimodule by $\pi(x) \theta(a^\op ) (\xi \otimes_N \eta)  := (\pi_H(x)\xi)\otimes_B (\theta_K(a^\op )\eta)$ for $x\in M$, $a\in A$, $\xi\in D(H, \varphi)$ and $\eta\in K$. 

	Since any standard representation $L^2(M)$ of $M$ is isomorphic with each other as $M$-bimodules, when we consider $H=K=L^2(M)$ and $N \subset M$, the Hilbert space $L^2(M)\otimes_N L^2(M)$ is determined canonically, and does not depend on the choice of a faithful normal semifinite weight $\varphi$ on $M$ with $L^2(M)=L^2(M,\varphi)$.

	Let $B \subset M$ be an inclusion of von Neumann algebras and $\varphi$ a faithful normal semifinite weight on $M$. The {\em basic construction} of the inclusion $B \subset M$ is defined by 
	$$\langle M, B\rangle := (J_\varphi B J_\varphi)' \cap \mathbb B(L^2(M,\varphi)).$$
Since all standard representations are canonically isomorphic, the basic construction does not depend on the choice of $\varphi$.
Assume that the inclusion $B \subset M$ is with an operator valued weight $E_B$. Fix a faithful normal semifinite weight $\varphi_B$ on $B$ and put $\varphi:=\varphi_B\circ E_B$. 
Here we observe that any $x\in \mathfrak{n}_{E_B} \cap \mathfrak{n}_\varphi$ is left $\varphi$-bounded and $L_{\Lambda_{\varphi}(x)}\Lambda_{\varphi_B}(a)=\Lambda_{\varphi}(xa)$ for $a\in \mathfrak{n}_{\varphi_B}$. 
Indeed, for any analytic $a\in \mathfrak{n}_{\varphi_B} \cap \mathfrak{n}_{\varphi_B}^*$, we have $J_{\varphi_B}\Lambda_{\varphi_B}(a^*) = \Delta_{\varphi_B}^{1/2}\Lambda_{\varphi_B}(a) = \Lambda_{\varphi_B}(\sigma_{-i/2}^{\varphi_B}(a))$ (e.g.\ the equation just before \cite[Lemma VIII.2.4]{Ta01}), and hence by \cite[Lemma V.III.3.18(ii)]{Ta01},
\begin{eqnarray*}
	L_{\Lambda_\varphi(x)}\Lambda_{\varphi_B}(\sigma_{-\frac{i}{2}}^{\varphi_B}(a)) 
	= L_{\Lambda_\varphi(x)}J_{\varphi_B}\Lambda_{\varphi_B}(a^*)
	= J_{\varphi} a^* J_\varphi\Lambda_{\varphi}(x) 
	= \Lambda_{\varphi}(x \sigma^{\varphi}_{-\frac{i}{2}}(a) ). 
\end{eqnarray*}
Since $\sigma_{-i/2}^{\varphi_B}(a)=\sigma_{-i/2}^{\varphi}(a)$ (because $\sigma^{\varphi}_t|_B=\sigma^{\varphi_B}_t$ for $t\in \R$, and the analytic extension is unique if exists), this means that $L_{\Lambda_\varphi(x)}\Lambda_{\varphi_B}(b)=\Lambda_{\varphi}(xb)$ for any analytic $b\in \mathfrak{n}_{\varphi_B} \cap \mathfrak{n}_{\varphi_B}^*$. 
At the same time, we can define a bounded operator $L_x\colon \Lambda_{\varphi_B}(a)\mapsto \Lambda_\varphi(xa)$ for $a\in \mathfrak{n}_{\varphi_B}$ (use $x\in \mathfrak{n}_{E_B}$). So the map $L_{\Lambda_\varphi(x)}$ has a bounded extension on $L^2(B,\varphi_B)$ and coincides with $L_x$, as desired. 
Now it is easy to verify that 
	$$L_{\Lambda_{\varphi}(y)}^*L_{\Lambda_{\varphi}(x)} = E_B(y^*x)\in (J_{\varphi}BJ_{\varphi})'=B \subset \B(L^2(B,\varphi_B)) \qquad (x,y\in \mathfrak{n}_{E_B}\cap \mathfrak{n}_\varphi).$$
We will use this formula for calculations in the proposition below and in  Section \ref{Weakly compact actions}.

	Here we observe that a relative tensor product has a useful identification. We will use this proposition in Sections \ref{Weakly compact actions} and \ref{Proof of Theorem A free quantum group factors}.

\begin{Pro}\label{lemma for relative tensor}
	Let $N$ and $B$ be von Neumann algebras, and $\alpha^N$ and $\alpha^B$ actions of $\R$ on $N$ and $B$ respectively. Put $M:=N\ovt B$ and $\alpha:=\alpha^N\otimes \alpha^B$, and define $H:=L^2(M\rtimes_{\alpha} \R) \otimes_B L^2(M\rtimes_{\alpha} \R)$ as an $M\rtimes_{\alpha}\R$-bimodule with left and right actions $\pi_H$ and $\theta_H$. 

Then there is a unitary $U \colon H \to L^2(\R)\otimes L^2(N) \otimes L^2(B) \otimes L^2(N)\otimes L^2(\R)$ such that, putting $\widetilde{\pi}_H:=\Ad U \circ \pi_H $ and $\widetilde{\theta}_H:=\Ad U \circ \theta_H $,
\begin{itemize}
	\item $\widetilde{\pi}_H( M\rtimes_{\alpha}\R  ) \subset \B(L^2(\R)\otimes L^2(N) \otimes L^2(B)) \otimes \C 1_N \otimes \C 1_{L^2(\R)} :$

		$\widetilde{\pi}_H(\lambda_t) = \lambda_t \otimes 1_N \otimes 1_B$ \quad and \quad $\widetilde{\pi}_H(x) = \pi_{\alpha}(x)$ \quad $(t\in \R$, $x\in N\ovt B =M)$;
	\item $\widetilde{\theta}_H( (M\rtimes_{\alpha}\R)^\op  ) \subset \C 1_{L^2(\R)} \otimes \C 1_N \otimes \B(L^2(B)\otimes L^2(N) \otimes L^2(\R)) :$

		$\widetilde{\theta}_H(\lambda_t^\op ) = 1_B \otimes 1_N \otimes \rho_t$ \quad and \quad $\widetilde{\theta}_H(y^\op ) = \theta_{\alpha}(y^\op )$ \quad $(t\in \R$, $y\in B\ovt N \simeq M)$,

		where $(\theta_{\alpha}(y^\op ) \xi)(s) := \alpha_{s}(y)^\op  \xi(s)$ for $\xi \in L^2(\R,L^2(B)\otimes L^2(N) )$ and $s\in \R$.
\end{itemize}
\end{Pro}
\begin{proof}
	We fix a faithful normal semifinite weight $\varphi_B$ on $B$ and put $\varphi:=\varphi_N\otimes \varphi_B$. Write as $\widehat{\varphi}$ the dual weight of $\varphi$ and then the standard representation of $M\rtimes_\alpha \R$ is given by 
	$$L^2(M\rtimes_\alpha \R, \widehat{\varphi})=L^2(N, \varphi_N)\otimes L^2(B,\varphi_B)\otimes L^2(\R)\simeq L^2(\R, L^2(N, \varphi_N)\otimes L^2(B,\varphi_B)).$$ 
For simplicity we put $L^2(N):=L^2(N,\varphi_N)$ and $L^2(B):= L^2(B,\varphi_B)$.  Let $E_B$ be the canonical operator valued weight from $\widetilde{M}$ to $B$ given by $\widehat{\varphi}=\varphi_B\circ E_B$. Then for $E_B^M:= \varphi_N\otimes \id_B$ on $M$ and for the canonical operator valued weight $E_M$ from $(M\rtimes \R, \widehat{\varphi})$ to $(M,\varphi)$, we have $\widehat{\varphi}=\varphi\circ E_M = \varphi_B \circ E_B^M \circ E_M$, and hence $E_B=E_B^M \circ E_M$ by the uniqueness condition. 
Observe then for any $f,g \in C_c (\R, M)$, 
	$$E_B(\widehat{\pi}_\alpha(g)^*\widehat{\pi}_\alpha(f)) = \int_\R E_B^M (g(t)^* f(t)) dt.$$
Define a well-defined linear map 
$$V \colon \Lambda_\varphi(\mathfrak{n}_{\varphi_N} \ota \mathfrak{n}_{\varphi_B}) \ota J_{\varphi} \Lambda_\varphi(\mathfrak{n}_{\varphi_B} \ota \mathfrak{n}_{\varphi_N}) \to L^2(N) \otimes L^2(B) \otimes L^2(N)$$ 
by $V (\Lambda_{\varphi}(x\otimes a) \otimes J_{\varphi} \Lambda_{\varphi}(b \otimes y)) := \Lambda_{\varphi_N}(x)\otimes aJ_{\varphi_B}\Lambda_{\varphi_B}(b) \otimes J_{\varphi_N}\Lambda_{\varphi_N}(y) $. 
We then define a linear map 
$$U \colon L^2(\R ,L^2(N) \otimes L^2(B)) \otimes_B L^2(\R ,L^2(B) \otimes L^2(N)) \to L^2(\R\times \R ,L^2(N) \otimes L^2(B) \otimes L^2(N))$$
by $(U (f\otimes_B J_{\widehat{\varphi}} g) )(t,s) := V (\Lambda_\varphi(f (t)) \otimes J_\varphi\Lambda_\varphi(g(-s))) $ for $f\in C_c (\R, N\ota B)(\mathfrak{n}_{\varphi_N} \ota \mathfrak{n}_{\varphi_B})$ and $g\in C_c (\R, B\ota N)(\mathfrak{n}_{\varphi_B} \ota \mathfrak{n}_{\varphi_N})$. (Note that we are identifying $\Lambda_{\widehat{\varphi}}(\widehat{\pi}_\alpha(f))$ and $\Lambda_{\widehat{\varphi}}(\widehat{\pi}_\alpha(g))$ as $f$ and $g$.) We have to show that it is a well-defined unitary map. For $f_i\in C_c (\R, N\ota B)(\mathfrak{n}_{\varphi_N} \ota \mathfrak{n}_{\varphi_B})$ and $g_i\in C_c (\R, B\ota N)(\mathfrak{n}_{\varphi_B} \ota \mathfrak{n}_{\varphi_N})$, straightforward but rather complicated computations yield, on the one hand,
\begin{eqnarray*}
	\|\sum_i f_i \otimes_B J_{\widehat{\varphi}}g_i\|_2^2 
	= \sum_{i,j} \int_\R \int_\R \langle  F_{j,i}  J_\varphi \Lambda_\varphi(g_i(-s)), J_\varphi \Lambda_\varphi(g_j(-s))  \rangle ds dt,
\end{eqnarray*}
where $F_{j,i}:=E_B^M(f_j(t)^*f_i(t))$, and on the other hand,
\begin{eqnarray*}
	&&\| U\sum_i  (f_i \otimes_B J_{\widehat{\varphi}}g_i)\|_2^2 \\
	&=& \sum_{i,j} \int_{\R\times \R}\langle V (\Lambda_\varphi(f_i(t)) \otimes J_\varphi \Lambda_\varphi(g_i(-s))), V (\Lambda_{\varphi}(f_j(t)) \otimes J_\varphi \Lambda_\varphi(g_i(-s))) \,  \rangle dtds.
\end{eqnarray*}
Hence if we show 
	$$\langle V (\Lambda_\varphi(x) \otimes J_\varphi \Lambda_\varphi(a)), V (\Lambda_\varphi(y) \otimes J_\varphi \Lambda_\varphi(b)) \,  \rangle = \langle  E_B^M(y^*x)  J_\varphi \Lambda_{\varphi}(a), J_\varphi \Lambda_{\varphi}(b)  \rangle$$
for any $x,y \in \mathfrak{n}_{\varphi_N}\ota \mathfrak{n}_{\varphi_B}$ and $a,b \in \mathfrak{n}_{\varphi_B}\ota \mathfrak{n}_{\varphi_N}$, then $U$ is a well-defined unitary map. However this equation follows easily if we put elementary elements. 

	Finally $L^2(\R \times \R, L^2(N) \otimes L^2(B) \otimes L^2(N))$ is canonically isomorphic to $L^2(\R) \otimes L^2(N) \otimes L^2(B) \otimes L^2(N) \otimes L^2(\R)$, where the first (resp.\ the second) variable in $\R \times \R$ corresponds to $L\R$ of the left one (resp.\ the right one) in the Hilbert space. It is then easy to see that $\widetilde{\pi}_H$ and $\widetilde{\theta}_H$ satisfy the desired condition.
\end{proof}

	Let $M$ and $N$ be von Neumann algebras, and let $H$ and $K$ be $M$-$N$-bimodules. We write as $\pi_H$ and $\theta_H$ (resp.\ $\pi_K$ and $\theta_K$) left and right actions on $H$ (resp.\ $K$). 
We say that \textit{$K$ is weakly contained in $H$}, denoted by $K \prec H$, if for any $\varepsilon>0$, finite subsets $\mathcal{E}\subset M$ and $\mathcal{F}\subset N$, and any vector $\xi \in K$, there are vectors $(\eta_i)_{i=1}^n \subset H$ such that 
	$$\left|\sum_{i=1}^n \langle \pi_H(x)\theta_H(y^\op )\eta_i , \eta_i \rangle_H - \langle \pi_K(x)\theta_K(y^\op )\xi , \xi \rangle_K \right| < \varepsilon \quad (x\in \mathcal{E}, \ y\in \mathcal{F}).$$
This is equivalent to saying that the algebraic $\ast$-homomorphism given by $\pi_H(x) \theta_H(y^\op ) \mapsto \pi_K(x) \theta_K(y^\op )$ for $x\in M$ and $y\in N$ is bounded on $\ast\text{-alg}\{\pi_H(M), \theta_H(N^\op ) \}$. We write as $\nu_{K,H}$ the associated $\ast$-homomorphism for $K\prec H$.

	Let $M$ and $N$ be \textit{$\sigma$-finite} von Neumann algebras and let $X$ be a self-dual $M$-$N$-correspondence (i.e.\ a Hilbert $N$-module with a normal left $M$-action, see \cite[Section 3]{Pa73} for self-duality and normality). Then the \textit{interior tensor product} (e.g.\ \cite[Section 4]{La95}) $H(X):=X\otimes_N L^2(N)$ is an $M$-$N$-bimodule. Conversely if $H$ is an $M$-$N$-bimodule, then one can define a self-dual $M$-$N$-correspondence (i.e.\ a W$^*$-Hilbert $N$-module with a left $M$-action)
$$X(H):= \{T\colon L^2(N) \to H \mid \text{bounded, $N^\op $-module linear map} \}.$$
They in fact give a one-to-one correspondence between $M$-$N$-bimodules and self-dual $M$-$N$-correspondences, up to unitary equivalence (see \cite[Theorem 2.2]{BDH88} and \cite[Proposition 6.10]{Ri74}). 
By \cite[\S1.12 PROPOSITION]{AD88}, $K \prec H$ if and only if $X(K) \prec X(H)$ in the following sense: for any $\sigma$-weak neighborhood $\mathcal{V}$ of $0\in N$, finite subsets $\mathcal{E}\subset M$ and $\mathcal{F}\subset N$, and any $\xi \in X(K)$, there are vectors $(\eta_i)_{i=1}^n \subset X(H)$ such that 
	$$\sum_{i=1}^n \langle \eta_i , x \eta_i y \rangle_{X(H)} - \langle \xi , x \xi y \rangle_{X(K)} \in \mathcal{V} \quad (x\in \mathcal{E}, \ y\in \mathcal{F}).$$

	Suppose that $M=N$, $L^2(M)=K$, and $M=X(K)$. Then if $L^2(M) \prec H$, putting $\xi:= 1_M$, for any finite subset $\mathcal{E}\subset M$ and for any $\sigma$-weak neighborhood $\mathcal{V}$ of $0\in N$, there are vectors $(\eta_i)_{i=1}^n \subset X(H)$ such that 
	$$\sum_{i=1}^n \langle \eta_i , x \eta_i \rangle_{X(H)} -  x  \in \mathcal{V} \quad (x\in \mathcal{E}).$$
So putting $\psi_{(\mathcal{E}, \mathcal{V})}(x):=\sum_{i=1}^n \langle \eta_i , x \eta_i \rangle_{X(H)}$ for $x\in M$, we find a net $(\psi_i)_i$ such that each $\psi_i$ is given by a sum of compression maps by vectors in $X(H)$ and that it converges to $\mathrm{id}_M$ in the point $\sigma$-weak topology. In this case, up to replacing $\eta_i$, we may assume that each $\psi_i$ is a contraction \cite[Lemma 2.2]{AH88}. Then it is known that the existence of such a net is equivalent to $L^2(M) \prec H$ as follows, although we do not need this equivalence. See \cite[Proposition 2.4]{AH88} for a more general statement.

\begin{Pro}\label{approximation of multiplication}
	Let $M$ be a $\sigma$-finite von Neumann algebra and  $H$ an $M$-bimodule. Then $L^2(M) \prec H$ as $M$-bimodules if and only if there is a net $(\psi_i)_i$ of normal c.c.p.\ maps on $M$, which converges to $\id_M$ point $\sigma$-weakly, such that each $\psi_i$ is a finite sum of $\langle \eta, \, \cdot \, \eta \rangle_{X(H)}$ for some $\eta\in X(H)$.
\end{Pro}

	We recall the following well-known fact. This will be used in Section \ref{Weakly compact actions}.

\begin{Lem}\label{correspondence with operator weight}
	Let $B\subset M$ be an inclusion of $\sigma$-finite von Neumann algebras with an operator valued weight $E_B$. 
Then the vector space $\mathfrak{n}_{E_B}$ is a pre-Hilbert $B$-module with the inner product $\langle x,y \rangle := E_B(x^*y)$ for $x,y \in\mathfrak{n}_{E_B}$, and its self-dual completion $\overline{\mathfrak{n}_{E_B}}$ is an $M$-$B$-correspondence. 

	Let $X$ be the self-dual completion of the interior tensor product $\overline{\mathfrak{n}_{E_B}} \otimes_B M$. Then as an $M$-$M$-correspondence, $X$ is the unique one corresponding to the $M$-bimodule $L^2(M)\otimes_B L^2(M)$, using the one-to-one correspondence above.
\end{Lem}
\begin{proof}
	It is easy to see that the $B$-valued inner product on $\mathfrak{n}_{E_B}$ in the statement is well-defined, so that $\mathfrak{n}_{E_B}$ is a pre-Hilbert $B$-module with a left $M$-action. 
Since the left $M$-action is faithful on $\mathfrak{n}_{E_B}$, so does on the self-dual completion (e.g.\ \cite[Corollary 3.7]{Pa73}). 
This left $M$-action is normal, since the functional $M \ni x \mapsto \omega (\langle \xi , x\eta \rangle )$ is normal for all $\omega\in M_*$ and $\xi,\eta \in \mathfrak{n}_{E_B}$, and hence for all $\xi,\eta \in \overline{\mathfrak{n}_{E_B}}$ by \cite[Lemma 2.3]{Pa75}. Thus $\overline{\mathfrak{n}_{E_B}}$ is an $M$-$B$-correspondence. 

	Let $X$ be as in the statement. Then as in the first paragraph, it is easy to see that it is really an $M$-$M$-correspondence (i.e.\ the left $M$-action is well-defined, injective, and normal). 
Let us fix faithful normal states $\varphi_B$ and $\varphi$ on $B$ and $M$ respectively. Then the interior tensor product $X\otimes_{M} L^2(M,\varphi)$ is canonically identified as $L^2(M, \varphi_B\circ E_M)\otimes_{B} L^2(M,\varphi)$, so that $X$ is identified as  $X(L^2(M)\otimes_{B}L^2(M))$. 
\end{proof}

\subsection{\bf Free quantum groups and bi-exactness}\label{Free quantum groups and bi-exactness}

	For compact quantum groups, we refer the reader to \cite{Wo95,MV98}.

	Let $\mathbb{G}$ be a compact quantum group. In this paper, we use the following notation, which will only be used in Section \ref{Proof of Theorem A free quantum group factors}. 
We denote the Haar state by $h$, the set of equivalence classes of all irreducible unitary corepresentations by $\mathrm{Irred}(\mathbb{G})$, and right and left regular representations by $\rho$ and $\lambda$ respectively. We regard $C_{\rm red}(\mathbb{G}):=\rho(C(\mathbb{G}))$ as our main object and we frequently omit $\rho$ when we see the dense Hopf $*$-algebra. 
The GNS representation of $h$ is written as $L^2(\mathbb{G})$ and it has a decomposition $L^2(\mathbb{G})=\sum_{x\in\mathrm{Irred}(\mathbb{G})}\oplus (H_x\otimes H_{\bar{x}})$. Along the decomposition, the modular operator of $h$ is of the form $\Delta_h^{it}=\sum_{x\in\mathrm{Irred}(\mathbb{G})}\oplus (Q_x^{it}\otimes Q_{\bar{x}}^{-it})$ for some positive matrices $Q_x$.

	Let $F$ be a matrix in $\mathrm{GL}(n,\mathbb{C})$. The \textit{free unitary quantum group} (resp.\ \textit{free orthogonal quantum group}) for $F$ \cite{Wa94,VW95} is the C$^*$-algebra $C(A_u(F))$ (resp.\ $C(A_o(F))$) defined as the universal unital C$^*$-algebra generated by all the entries of a unitary $n$ by $n$ matrix $u=(u_{i,j})_{i,j}$ satisfying that $F(u_{i,j}^*)_{i,j}F^{-1}$ is a unitary (resp.\ $F(u_{i,j}^*)_{i,j}F^{-1}=u$). We simply say that $\G$ is a \textit{free quantum group} if $\G$ is a free unitary or orthogonal quantum group. 

	Here we recall the notion of bi-exactness introduced in \cite[Definition 3.1]{Is13}, based on the group case \cite[Lemma 15.1.2]{BO08}.

\begin{Def}\upshape\label{bi-exact definition}
	Let $\mathbb{G}$ be a compact quantum group. We say that the dual $\widehat{\mathbb{G}}$ is \textit{bi-exact} if it satisfies following conditions:
\begin{itemize}
	\item[$\rm (i)$] $\widehat{\mathbb{G}}$ is exact (i.e.\ $C_{\rm red}(\mathbb{G})$ is exact);
	\item[$\rm (ii)$]  there exists a u.c.p.\ map $\Theta\colon C_{\rm red}(\mathbb{G})\otimes_{\rm min}  C_{\rm red}(\mathbb{G})^\op  \rightarrow \mathbb{B}(L^2(\mathbb{G}))$ such that 
	$$\Theta(a\otimes b^\op )-ab^\op \in \mathbb{K}(L^2(\mathbb{G})),  \quad \text{for any } a,b \in  C_{\rm red}(\mathbb{G}).$$
\end{itemize}
\end{Def}

	Bi-exactness of free quantum groups were proved in \cite{Ve04,VV05,VV08}. See \cite[Theorem C]{Is13} for other examples of bi-exact quantum groups.

\begin{Thm}\label{properties of free quantum group}
	Let $\G$ be a free quantum group (more generally, a compact quantum group in \cite[Theorem C]{Is13}). Then the dual $\widehat{\G}$ is bi-exact.
\end{Thm}

\subsection{\bf Central weak amenability and the W$^*$CMAP}\label{WCMAP}

	Let $\G$ be a compact quantum group. Denote the dense Hopf $*$-algebra by $\mathscr{C}(\G)$. 
For any element $a\in\ell^\infty(\widehat{\mathbb{G}})$, we can associate a linear map $m_a$ on $\mathscr{C}(\G)$, given by $(m_a\otimes \iota)(u^x)=(1\otimes ap_x)u^x$ for any $x\in \mathrm{Irred}(\mathbb{G})$, where $p_x\in c_0(\widehat{\mathbb{G}})$ is the canonical projection onto $x$ component. 
We say $\widehat{\mathbb{G}}$ is \textit{weakly amenable (with Cowling--Haagerup constant 1)} if there exist a net $(a_i)_i$ of elements of $\ell^\infty(\widehat{\mathbb{G}})$ such that  
\begin{itemize}
	\item each $a_i$ has finite support, namely, $a_ip_x=0$ except for finitely many $x\in \mathrm{Irred}(\mathbb{G})$;
	\item $(a_i)_i$ converges to 1 pointwise, namely, $a_ip_x$ converges to $p_x$ in $\mathbb{B}(H_x)$ for any $x\in \mathrm{Irred}(\mathbb{G})$;
	\item each $m_{a_i}$ is extended on $L^\infty(\G)$ as a completely contractive (say c.c.) map.
\end{itemize}
Note that, since $a_i$ is finitely supported, each $m_{a_i}$ is actually a map from $L^\infty(\G)$ to $\mathscr{C}(\G)$. We say $\widehat{\G}$ is \textit{centrally weakly amenable} if each $a_ip_x$ above is taken as a scalar matrix for all $i$ and $x\in \mathrm{Irred}(\G)$. In this case, the associated multiplier $m_{a_i}$ commutes with the modular action of the Haar state. This commutativity is important to us since such multipliers can be extended naturally on the continuous core with respect to the Haar state. Indeed, the maps $m_{a_i}\otimes \id_{L^2(\R)}$ on $L^\infty(\G)\ovt \B(L^2(\R))$ restrict to approximation maps on the core. 
With this phenomenon in our mind, we introduce the following terminology.

\begin{Def}\upshape\label{WCMAP with respect to a state}
	Let $M$ be a von Neumann algebra and $\varphi$ a fixed faithful normal state  on $M$. We say that $M$ has the $\it weak^*$ \textit{completely metric approximation property with respect to $\varphi$} (or \textit{$\varphi$-W$^*$CMAP}, in short) 
if there exists a net $(\psi_i)_i$ of normal c.c.\ maps on $M$ such that
\begin{itemize}
	\item each $\psi_i$ commutes with $\sigma^\varphi$, that is, $\psi_i \circ \sigma^\varphi_t = \sigma_t^\varphi \circ \psi_i$ for all $i$ and $t\in \R$;
	\item each $\psi_i$ is a finite sum of $\varphi(b^* \, \cdot \, a)z$ for some $a,b,z \in M$;
	\item $\psi_i$ converges to $\mathrm{id}_M$ in the point $\sigma$-weak topology.
\end{itemize}
\end{Def}
It is easy to see that the central weak amenability of $\widehat{\G}$ implies the W$^*$CMAP with respect to the Haar state.

	Weak amenability of the free quantum group was first obtained in \cite{Fr12}, using the Haagerup property \cite{Br11}. This is for the Kac type and hence is equivalent to the central weak amenability. The general case was solved later in \cite{DFY13} and its proof in fact shows the central weak amenability as follows.

\begin{Thm}\label{weakly amenable}
	Let $\G$ be a free quantum group (more generally a quantum group in \cite[Theorem C]{Is13}). Then the dual $\widehat{\G}$ is centrally weakly amenable. 

In particular there is a net $(\psi_i)_i$ of normal c.c.\ maps on $L^\infty(\G)$, witnessing the W$^*$CMAP with respect to the Haar state, such that $\psi_i(L^\infty(\G)) \subset \mathscr{C}(\G)$ for all $i$.
\end{Thm}

\subsection{\bf Popa's intertwining techniques}\label{Popa's intertwining techniques}

	In \cite{Po01,Po03}, Popa introduced a powerful tool called \textit{intertwining techniques}. This is one of the main ingredient in the recent development of the von Neumann algebra theory. Here we introduce the one defined and studied in \cite[Definition 4.1 and Theorem 4.3]{HI15} which treats  general von Neumann algebras. 

\begin{Def}\upshape\label{definition intertwining}
	Let $M$ be any $\sigma$-finite von Neumann algebra, $1_A$ and $1_B$ any nonzero projections in $M$, $A\subset 1_AM1_A$ and $B\subset 1_BM1_B$ any von Neumann subalgebras with expectation. 
We say that $A$ {\em embeds with expectation into} $B$ {\em inside} $M$ and write $A \preceq_M B$ if there exist projections $e \in A$ and $f \in B$, a nonzero partial isometry $v \in eMf$ and a unital normal $\ast$-homomorphism $\theta : eAe \to fBf$ such that the inclusion $\theta(eAe) \subset fBf$ is with expectation and $av = v \theta(a)$ for all $a \in eAe$.
\end{Def}

\begin{Thm}\label{intertwining thm}
	Keep the same notation as in Definition \ref{definition intertwining} and assume that $A$ is finite. Then the following conditions are equivalent.
	\begin{itemize}
		\item[$(1)$] We have $A \preceq_M B$.
		\item[$(2)$] There exists no net $(w_i)_{i \in I}$ of unitaries in $\mathcal U(A)$ such that $E_{B}(b^*w_i a)\rightarrow 0$ in the $\sigma$-$\ast$-strong topology for all $a,b\in 1_AM1_B$, where $E_B$ is a fixed faithful normal conditional expectation from $1_BM1_B$ onto $B$.
		\end{itemize}
\end{Thm}

	For the proof of Corollary \ref{corB}, we prove a lemma. In the proof below, we make use of the \textit{ultraproduct} von Neumann algebras  \cite{Oc85}. We will actually use a more general one used in \cite{HI15}, which treats a general directed set instead of $\N$. Recall from \cite[Section 2]{HI15} that for any $\sigma$-finite von Neumann algebra $M$ and any free ultrafilter $\mathcal{U}$ on a directed set $I$, we may define the \textit{ultraproduct von Neumann algebra $M^{\mathcal{U}}$}, using $\ell^\infty(I)\ovt M$. 
In the proof below, we only need the following elementary properties: with the standard notation $(x_i)_\mathcal{U} \in M^{\mathcal{U}}$ for $(x_i)_{i\in I}$,
\begin{itemize}
	\item $M \subset M^{\mathcal{U}}$ is with expectation by $E_{\mathcal{U}}((x_i)_{\mathcal{U}} ):=\lim_{i\to \mathcal{U}} x_i$;
	\item for any $\sigma$-finite von Neumann algebras $A\subset M$ with expectation $E_A$, $A^{\mathcal{U}}\subset M^{\mathcal{U}}$ is with expectation defined by $E_{A^{\mathcal{U}}}((x_i)_{\mathcal{U}}):=(E_A(x_i))_{\mathcal{U}}$ ; 
	\item if the subalgebra $A$ is finite, then any norm bounded net $(a_i)_{i\in I}$ determines an element $(a_i)_{\mathcal{U}}$ in $M^{\mathcal{U}}$.
\end{itemize}

\begin{Lem}\label{intertwining lemma}
	Let $(B,\varphi_B)$ and $(N,\varphi_N)$ be von Neumann algebras with faithful normal states. Put $M:=B\ovt N$, $\varphi:=\varphi_B\otimes \varphi_N$, $E_B=\id_B\otimes \varphi_N$ and $E_N=\varphi_B\otimes \id_N$. Let $p\in M$ be a projection and $A\subset pMp$ a von Neumann subalgebra with expectation. Fix $a:=(a_i)_{i\in I} \in \ell^\infty(I)\ovt A$ and a free ultrafilter $\mathcal{U}$ on $I$ such that $(a_i)_{\mathcal{U}}\in A^{\mathcal{U}}$. 
Then $E_{B^{\mathcal{U}}}(y^* a x)=0$ for all $x,y\in M$ if and only if $E_N\circ E_{\mathcal{U}}(c^* a b)$ for all $b,c\in B^{\mathcal{U}}$.

	In particular, if $A$ is finite, then $A\preceq_M B$ if and only if $A \preceq_{B\ovt N_0} B$ for any $N_0\subset N$ with expectation $E_{N_0}$ such that $\varphi_N\circ E_{N_0}=\varphi_N$, $p\in B\ovt N_0$ and $A\subset p(B\ovt N_0)p$.
\end{Lem}
\begin{proof}
	Observe first that $E_{B^{\mathcal{U}}}(y^* a x)=0$ for all $x,y\in M$ if and only if $E_{B^{\mathcal{U}}}((1\otimes y^*) a (1\otimes x))=0$ for all $x,y\in N$, which is equivalent to 
$$\langle E_{B^{\mathcal{U}}}((1\otimes y^*) a (1\otimes x)) \Lambda_{\varphi_{B}^{\mathcal{U}}}(b), \Lambda_{\varphi_{B}^{\mathcal{U}}}(c) \rangle_{\varphi^{\mathcal{U}}_B}=0$$ 
for all $x,y\in N$ and $b,c \in B^{\mathcal{U}}$. Writing as $b=(b_i)_{\mathcal{U}}$ and $c=(c_i)_{\mathcal{U}}$, calculate that 
\begin{eqnarray*}
	&& \langle E_{B^{\mathcal{U}}}((1\otimes y^*) a (1\otimes x)) \Lambda_{\varphi_{B}^{\mathcal{U}}}(b), \Lambda_{\varphi_{B}^{\mathcal{U}}}(c) \rangle_{\varphi^{\mathcal{U}}_B} \\
	&=&\lim_{i\to {\mathcal{U}}} \langle E_{B}((1\otimes y^*) a_i (1\otimes x)) \Lambda_{\varphi_{B}}(b_i), \Lambda_{\varphi_{B}}(c_i) \rangle_{\varphi_B} \\
	&=&\lim_{i\to {\mathcal{U}}} \varphi_B(c_i^* E_{B}((1\otimes y^*) a_i (1\otimes x))b_i) \\
	&=&\lim_{i\to {\mathcal{U}}} \varphi_B\circ E_B( (c_i^*\otimes y^*) a_i (b_i\otimes x))\\ 
	&=&\lim_{i\to {\mathcal{U}}}  \varphi_N \circ E_N ( (c_i^*\otimes y^*) a_i (b_i\otimes x)) \\
	&=&\lim_{i\to {\mathcal{U}}}  \varphi_N( y^*  E_N ( (c_i^*\otimes 1) a_i (b_i\otimes 1)) x) \\
	&=& \varphi_N( y^*  E_N (\lim_{i\to {\mathcal{U}}}( (c_i^*\otimes 1) a_i (b_i\otimes 1))) x) \\
	&=& \varphi_N( y^*  E_N \circ E_{\mathcal{U}}( (c^*\otimes 1) a (b\otimes 1)) x).
\end{eqnarray*}
Then since functionals of the form $\varphi_N(y^* \cdot x)$ for $x,y\in N$ are norm dense in $N_*$, the final term above is zero for all $x,y\in N$ if and only if $E_N \circ E_{\mathcal{U}}( (c^*\otimes 1) a (b\otimes 1))=0$. 
Thus we proved that $E_{B^{\mathcal{U}}}(y^* a x)=0$ for all $x,y\in M$ if and only if $E_N \circ E_{\mathcal{U}}( (c^*\otimes 1) a (b\otimes 1))=0$ for all $b,c \in B^{\mathcal{U}}$. 

	For the second half of the statement, suppose that $A$ is finite and $A\not\preceq_{B\ovt N_0}B$. We will show $A\not\preceq_MB$. Since $A$ is finite, there is a net $(u_i)_{i\in I}\subset \mathcal{U}(A)$ for a directed set $I$ such that $E_B(y^*u_ix)\to 0$ strongly as $i\to \infty$ for all $x,y\in B\ovt N_0$. Fix any co-finial ultrafilter ${\mathcal{U}}$ on $I$. Since $A$ is finite, $u:=(u_i)_{\mathcal{U}} \in A^{\mathcal{U}}$ and hence $E_{B^{\mathcal{U}}}(y^* u x)=0$ for all $x,y\in  B\ovt N_0$. 
By the first half of the statement, this is equivalent to $E_{N_0}\circ E_{\mathcal{U}}(c^* u b)=0$ for all $b,c\in B^{\mathcal{U}}$. Then since $E_{\mathcal{U}}(c^* u b)$ is contained in $B\ovt N_0$ and since $E_{N}|_{B\ovt N_0}= (\varphi_B \otimes \id_N)|_{B\ovt N_0}=E_{N_0}$, we have $E_{N}\circ E_{\mathcal{U}}(c^* u b)=0$ for all $b,c\in B^{\mathcal{U}}$, which is in turn equivalent to $E_{B^{\mathcal{U}}}(y^* u x)=0$ for $x,y \in M$ by the first half of the statement. Since this holds for arbitrary ${\mathcal{U}}$ on $I$, we conclude that $E_B(y^*u_ix)\to 0$ $\ast$-strongly as $i \to \infty$ for all $x,y\in M$. 
Thus we proved that $A\not\preceq_{B\ovt N_0}B$ implies $A\not\preceq_{M}B$.
\end{proof}

\section{\bf Weakly compact actions}\label{Weakly compact actions}

	In this section, we define and study weakly compact actions on continuous cores. The main observation is Theorem \ref{weakly compact action on core}, and the key item for the proof is Lemma \ref{key lemma for approximation1}.

\subsection{\bf Relative amenability and approximation maps}

	In this subsection, we recall relative amenability for general von Neumann algebras introduced in \cite{Is17}, which generalizes \cite{OP07} and $\cite{PV11}$.

\begin{Def}\label{relative amenable def1}\upshape
	Let $B\subset M$ be von Neumann algebras, $p\in M$ a projection and $A\subset pMp$ a von Neumann subalgebra with expectation $E_A$. 
We say that the pair \textit{$(A,E_A)$ is injective relative to $B$ in $M$}, and write as $(A,E_A)\lessdot_MB$, if there exists a conditional expectation from $p\langle M,B\rangle p$ onto $A$ which restricts to $E_A$ on $pMp$.
\end{Def}

	Using amenability of $\R$ and the notion of relative amenability, we prove a lemma for approximation maps on the continuous core. 
For this we fix the following notation. 

Let $(M,\varphi)$ be a von Neumann algebra with a faithful normal semifinite weight, and $\widetilde{M}:=M\rtimes \R$ the continuous core of $M$ with the modular action $\sigma^\varphi$. 
We write as $\widehat{\varphi}$ the dual weight of $\varphi$, and as $E_M$ the canonical operator valued weight from $\widetilde{M}$ to $M$ given by $\widehat{\varphi}=\varphi\circ E_M$. 
We write as $M\rtimes_{\rm alg} G$ all the linear spans of $x\lambda_t$ for $x\in M$ and $t\in G$, which is a $\ast$-strongly dense subalgebra in $\widetilde{M}$.

\begin{Lem}\label{weakly contained for amenable crossed product}
	In this setting, we have 
$${}_{\widetilde{M}}L^2(\widetilde{M})_{\widetilde{M}} \prec {}_{\widetilde{M}} L^2(\widetilde{M})\otimes_M L^2(\widetilde{M})_{\widetilde{M}}.$$
\end{Lem}
\begin{proof}
	Recall first that 
	$$M\rtimes \R = (M^\op \otimes 1)' \cap \{\Delta_{\varphi}^{it}\otimes \rho_t \mid t\in \R\}', \quad  \langle M\rtimes \R,M \rangle = (M^\op \otimes 1)',$$
where $\rho$ is the right regular representation. 
Since $\R$ is amenable, there is positive functionals $(f_n)_n\subset L^1(\R)$ with $\|f_n\|_1=1$, satisfying $\lambda_gf_n-f_n\to 0$ weakly for all $g\in \R$. For each $n$, define a positive map $F_n\colon \B(L^2(M)\otimes L^2(\R)) \to \B(L^2(M)\otimes L^2(\R))$ by 
$$F_n(T):=\int_\R (\Delta_{\varphi}^{it}\otimes \rho_t) \, T \, (\Delta_{\varphi}^{it}\otimes \rho_t)^* f_n(t) \cdot dt .$$
Since $\|F_n\|= 1$, we can take a cluster point of $(F_n)_n$, which we write as $F$. Then it satisfies $(\Delta_{\varphi}^{it}\otimes \rho_t)F(T)(\Delta_{\varphi}^{it}\otimes \rho_t)^*=F(T)$ for all $t\in \R$ and hence $F$ is a conditional expectation onto $\{\Delta_{\varphi}^{it}\otimes \rho_t\mid t\in \R\}'$. It is easy to see that $F(T)\in (M^\op \otimes 1)'$ for any $T\in (M^\op \otimes 1)'$. Hence $F$ restricts to a conditional expectation from $\langle M\rtimes \R,M \rangle$ onto $M\rtimes \R$. We obtain $(M\rtimes \R, \id)\lessdot_{M\rtimes \R} M$. 
Finally since $M\rtimes \R$ is semifinite, using \cite[Theorem A.5]{Is17}, we get the conclusion.
\end{proof}

\begin{Lem}\label{key lemma for approximation1}
	In this setting, there is a net $(\omega_j)_j$ of c.c.p.\ maps on $\widetilde{M}$ such that $\omega_j \to \id_{\widetilde{M}}$ point $\sigma$-weakly and each $\omega_j$ is a finite sum of $\lambda_q^*E_M(z^* \, \cdot \, y )\lambda_p$ for some $y,z\in \mathfrak{n}_{E_M}$ and $p,q\in \R$.
\end{Lem}
\begin{proof}
	By Corollary \ref{weakly contained for amenable crossed product} and Proposition \ref{approximation of multiplication}, there is a net $(\omega_j)_j$ of c.c.p.\ maps on $\widetilde{M}$ such that $\omega_j \to \id_{\widetilde{M}}$ point $\sigma$-weakly and each $\omega_j$ is a finite sum of $\langle \eta, \, \cdot \, \eta \rangle_{X(L^2(\widetilde{M})\otimes_{M}L^2(\widetilde{M}))}$ for some $\eta\in X(L^2(\widetilde{M})\otimes_{M}L^2(\widetilde{M}))$. We first replace each $\eta$ in $\omega_j$ with some ``algebraic'' element in $X(L^2(\widetilde{M})\otimes_{M}L^2(\widetilde{M}))$.

	By Lemma \ref{correspondence with operator weight}, the self dual completion $X$ of $\overline{\mathfrak{n}_{E_{M}}}\otimes_{\rm alg} \widetilde{M}$ is identified as the one corresponding to $L^2(\widetilde{M})\otimes_{M}L^2(\widetilde{M})$. We denote by $X_0$ the image of $\overline{\mathfrak{n}_{E_{M}}}\otimes_{\rm alg} \widetilde{M}$ in $X$. 
By \cite[Lemma 2.3]{Pa75}, $X_0\subset X$ is dense in the \textit{s-topology}, that is, for any $\eta\in X$ there is a net $(\eta_i)_i \subset X_0$ such that $\langle \eta-\eta_i, \eta-\eta_i\rangle_{X} \to 0$ in the $\sigma$-weak topology in $\widetilde{M}$. 
In our case, since $\mathfrak{n}_{E_B} \subset \overline{\mathfrak{n}_{E_B}}$ is dense in the s-topology and since $M\rtimes_{\rm alg} G \subset \widetilde{M}$ is $\ast$-strongly dense, the image of $\mathfrak{n}_{E_{M}}\otimes_{\rm alg} (M\rtimes_{\rm alg} G)$ in $X$ is dense in the s-topology. Hence we may replace each vector $\eta \in X$, appearing in $\omega_j$ above, with the one represented by elements in $\mathfrak{n}_{E_{M}}\otimes_{\rm alg} (M\rtimes_{\rm alg} G)$. 

	Thus, we may assume that each $\omega_j$ is a finite sum of $\lambda_q^*E_M(z^* \, \cdot \, y )\lambda_p$ for some $y,z\in \mathfrak{n}_{E_M}$ and $p,q\in \R$. However the c.b.\ norms of the resulting net $(\omega_j)_j$ is no longer uniformly bounded. So we have to again replace $(\omega_j)_j$ with c.c.p.\ maps. 
For this, we assume that, up to convex combinations, the convergence $\omega_j\to \id_{\widetilde{M}}$ is in the point strong topology. 

	Recall from (the first half of) the proof of \cite[Lemma 2.2]{AD88} that if we put $\varphi_i(x):= c_j \omega_j(x) c_j$ for $x\in \widetilde{M}$, where $c_j:=2(1+\omega_j(1))^{-1}$, then the net $(\varphi_i)_i$ satisfies that each $\varphi_i$ is c.c.p.\ and that $\varphi_i \to \id_{\widetilde{M}}$ in the point strong topology. We will replace $c_j$ with elements in $M\rtimes_{\rm alg} G$. 
For this, fix $j$ and observe that, since $1+\omega_j(1)$ is in $M\rtimes_{\rm alg} G$, each $c_j$ is actually contained in $\mathrm{C}^*\{ M\rtimes_{\rm alg} G\}$, which is the norm closure of $M\rtimes_{\rm alg}G$. So there is a sequence $(a_n)_n$ in $M\rtimes_{\rm alg} G$ such that $\|a_n\|_\infty \leq \|c_j^{1/2}\|_\infty$ and $\|a_n - c_j^{1/2}\|_\infty \to 0$. Put $b_n := a_n^*a_n \in M\rtimes_{\rm alg}G$ and observe that it satisfies $\|b_n\|_\infty \leq \|c_j\|_\infty$ and $\|b_n - c_j\|_\infty \to 0$. 
It then holds that for any $x\in \widetilde{M}$,
\begin{eqnarray*} 
	\|c_j \omega_j(x) c_j - b_n \omega_j(x) b_n\|_\infty 
	\leq 2\|c_j \|_\infty \|\omega_j\|_{\rm cb} \|x\|_\infty \| c_j - b_n\|_\infty 
	\to 0, \quad \text{ as } n\to \infty. 
\end{eqnarray*}
Now fix any $\varepsilon>0$ and finite subset $\mathcal{F}\subset (\widetilde{M})_1$ such that $1\in \mathcal{F}$, and choose $b_n$ such that $\|c_j \omega_j(x) c_j - b_n \omega_j(x) b_n\|_\infty < \varepsilon$ for all $x\in \mathcal{F}$. Then since $1\in \mathcal{F}$, we have 
	$$\|b_n \omega_j(\cdot)b_n\|_{\rm cb} = \|b_n \omega_j(1)b_n\|_\infty < \|c_j \omega_j(1)c_j\|_\infty+ \varepsilon \leq 1+\varepsilon. $$
So $(1+\varepsilon)^{-1} b_n \omega_j(\cdot) b_n$ is a c.c.p.\ map which is still close to $c_j \omega_j(\cdot) c_j$ on $\mathcal{F}$. 
Thus we proved that for any $j$ there is a net of c.c.p.\ maps conversing to $c_j \omega_j(\cdot)c_j$ in the point \textit{norm} topology such that each map is a finite sum of $\lambda_q^*E_M(z^* \, \cdot \, y )\lambda_p$ for some $y,z\in \mathfrak{n}_{E_M}$ and $p,q\in G$. 
Using this observation, since $c_j\omega_j(\cdot)c_j\to \id_{\widetilde{M}}$ as $j\to \infty$ in the point strong topology, it is easy to construct a desired net.
\end{proof}

\subsection{\bf Definition of weakly compact actions}\label{Definition}

	We introduce a terminology. The following notion is an appropriate generalization of \cite[Definition 3.1]{OP07} in our setting (see also \cite[Theorem 5.1]{PV11}). Indeed, in the definition below, if we take $\mathcal{M}=M\ovt M^\op $, this coincides with the original definition of weakly compact actions.

\begin{Def}\upshape\label{def relative weakly compact action}
	Let $M$ be a semifinite von Neumann algebra with trace $\Tr$, and let $\mathcal{M}$ be a von Neumann algebra which contains $M$ and $M^\op $ as von Neumann subalgebras, which we denote by $\pi(M)$ and $\theta(M^\op )$, such that $[\pi(M), \theta(M^\op )]=0$. 

	Let $p\in M$ be a projection with $\Tr(p)=1$, $A \subset pMp$ be a von Neumann subalgebra, $\mathcal{G} \leq \mathcal{N}_{pMp}(A)$ a subgroup. We say that the adjoint action of $\mathcal{G}$ on $A$ is \textit{weakly compact for  $(M, \Tr, \pi, \theta, \mathcal{M})$} if 
there is a net $(\xi_i)_i$ of unit vectors in the positive cone of $L^2(\mathcal{M})$  such that 
\begin{itemize}
	\item[$\rm (i)$] $ \langle \pi(x) \xi_i , \xi_i \rangle_{L^2(\mathcal{M})} \rightarrow \Tr(pxp)$, \quad for any $x\in M$;
	\item[$\rm (ii)$] $ \| \pi(a)\theta(\bar{a}) \xi_i - \xi_i \|_{L^2(\mathcal{M})} \rightarrow 0$, \quad for any $a\in \mathcal{U}(A)$;
	\item[$\rm (iii)$] $\|\pi(u)\theta(\bar{u}) \mathcal{J}_{\mathcal{M}}\pi(u)\theta(\bar{u})\mathcal{J}_{\mathcal{M}} \xi_i - \xi_i\|_{L^2(\mathcal{M})} \rightarrow 0$, \quad for any $u \in \mathcal{G}$. 
\end{itemize}
Here $\bar a$ means $(a^\op )^*$ and $\mathcal{J}_{\mathcal{M}}$ is the modular conjugation for $L^2(\mathcal{M})$.
\end{Def}
\begin{Rem}\upshape\label{rem relative weakly compact action}
	In this definition, since $\mathcal{J}_{\mathcal{M}}\xi_i=\xi_i$ for all $i$, condition (ii) for $a\in \mathcal{U}(A)$ implies condition (iii) for $a \in \mathcal{U}(A)$. Hence up to replacing $\mathcal{G}$ with the group generated by $\mathcal{U}(A)$ and $\mathcal{G}$, we may always assume that $\mathcal{G}$ contains $\mathcal{U}(A)$.
\end{Rem}

	Below we record a characterization for weakly compact actions.

\begin{Pro}\label{lem relative weakly compact action}
	Keep the notation in Definition \ref{def relative weakly compact action}. The following conditions are equivalent.
\begin{itemize}
	\item[$(1)$] The group $\mathcal{G}$ acts on $A$ as a weakly compact action for $(M, \Tr, \pi, \theta, \mathcal{M})$.
	\item[$(2)$] There exists a net $(\omega_i)_i$ of normal states on $\mathcal{M}$ such that
	\begin{itemize}
		\item[$\rm (i)$] $\omega_i(\pi(x))\rightarrow \Tr(pxp)$, \quad for any $x\in pMp$;
		\item[$\rm (ii)$] $\omega_i(\pi(a)\theta(\bar{a}))\rightarrow 1$, \quad for any $a\in \mathcal{U}(A)$;
		\item[$\rm (iii)$] $\|\omega_i\circ\mathrm{Ad}(\pi(u)\theta(\bar{u}))-\omega_i\| \rightarrow 0$, \quad for any $u \in \mathcal{G}$. 
	\end{itemize}
	\item[$(3)$] There is a $\mathcal{G}$-central state $\omega$ on $\mathcal{M}$ such that for any $x\in M$ and $a\in \mathcal{U}(A)$
$$\omega(x)= \Tr (p x p) \quad \text{and} \quad  \omega(\pi(a)\theta(\bar{a}))= 1.$$ 
	\item[$(4)$] There is a state $\Omega$ on $\B(L^2(\mathcal{M}))$ such that for any $x\in M$, $a\in \mathcal{U}(A)$ and $u \in \mathcal{G}$, $$\Omega(x)= \Tr (p x p), \quad \Omega(\pi(a)\theta(\bar{a}))= 1, \quad \text{and} \quad \Omega((\pi(u)\theta(\bar{u})\mathcal{J}_{\mathcal{M}}\pi(u)\theta(\bar{u})\mathcal{J}_{\mathcal{M}})=1.$$ 
\end{itemize}
\end{Pro}
\begin{proof}
	This theorem follows from well-known arguments (e.g.\ the proof of \cite[Theorem 2.1]{OP07}). So we give a sketch of proofs. 

If (1) holds, then one put $\Omega:=\mathrm{Lim}_i\langle \,\cdot \, \xi_i, \xi_i \rangle_{L^2(\mathcal{M})}$ and obtain (4). If (4) holds, then the restriction of $\Omega$ on $\mathcal{M}$ gives (3). If (3) holds, then we can approximate $\omega$ by a net of normal states $(\omega_i)_i\subset \mathcal{M}_*$ weakly. Then by the Hahn--Banach separation theorem, up to convex combinations, we may assume that the convergence is in the norm and obtain (2). Finally if (2) holds, then for each $i$ one can find a unique $\xi_i\in L^2(\mathcal{M})$ which is in the positive cone such that $\omega_i= \langle \,\cdot \, \xi_i, \xi_i \rangle_{L^2(\mathcal{M})}$. By the Powers--St\o rmer inequality \cite[Theorem IX.1.2(iv)]{Ta01}, we obtain 
	$$\| \pi(u)\theta(\bar{u})\mathcal{J}_{\mathcal{M}}\pi(u)\theta(\bar{u})\mathcal{J}_{\mathcal{M}}\xi_i - \xi_i\|^2 \leq \|\omega_i\circ\mathrm{Ad}(\pi(u^*)\theta(u^\op ))-\omega_i\| \rightarrow 0$$ 
for any $u \in \mathcal{G}$ and hence (1) holds.
\end{proof}

\subsection{\bf W$^*$CMAP with respect to a state produces approximation maps on continuous cores}\label{WCMAP with respect to a state produces approximation maps on continuous cores}

	We construct a family of approximation maps on continuous cores, by assuming the W$^*$CMAP with respect to a state.

	For this, we fix the following setting. Let $N$ and $B$ be von Neumann algebras and  $\varphi_N$ and $\varphi_B$ faithful normal states on $N$ and $B$ respectively. Put $M:=N\ovt B$, $\varphi:=\varphi_N \otimes \varphi_B$, $E_N:=\id_N\otimes \varphi_B$ and $E_B:=\varphi_N\otimes \id_B$, and we regard $\widetilde{B}:=B\rtimes_{\sigma^{\varphi_B}}\R$ and $\widetilde{N}:=N\rtimes_{\sigma^{\varphi_N}}\R$ as subalgebras of $\widetilde{M}:=M\rtimes_{\sigma^{\varphi}}\R$. 
We write as $E_M$ the canonical operator valued weight from $\widetilde{M}$ to $M$ given by $\widehat{\varphi}=\varphi\circ E_M$, where $\widehat{\varphi}$ is the dual weight on $\widetilde{M}$. 
We also write as $E_B$ the canonical operator valued weight from $\widetilde{M}$ to $B$ given by $\widehat{\varphi}=\varphi_B\circ E_B$. 

\begin{Lem}\label{key lemma for approximation2}
	Let $\omega\colon \widetilde{M} \to \widetilde{M}$ and $\psi\colon N \to N$  be c.b.\ maps given by 
$$\omega:=\lambda_q^*E_M(z^* \, \cdot \, y)\lambda_p \quad \text{and} \quad \psi:=\sum_{i=1}^n \varphi_N(z_i^* \, \cdot \, y_i) c_i$$
for some $p,q\in\R$, $y,z \in \mathfrak{n}_{E_M}$ and $c_i, y_i, z_i\in N$. 
Suppose $\psi\circ \sigma_{t}^{\varphi_N}=\sigma_{t}^{\varphi_N}\circ \psi$ for all $t\in \R$, so that the map $\widetilde{\psi}:= \psi \otimes\id_B\otimes \id_{L^2(\R)}$ on $M\ovt \B(L^2(\R))$ induces the map $\widetilde{M}\to \widetilde{M}$ given by $\widetilde{\psi}(x \lambda_t)=(\psi\otimes\id_B)(x) \lambda_t$ for $x\in M$ and $t\in \R$. 
Then the composition $\widetilde{\psi}\circ \omega$ is given by
$$\widetilde{\psi}\circ \omega (x) = \sum_{i=1}^n \lambda_q^* E_B(\sigma_q^{\varphi_N}(z_i^*) z^* x y \sigma_p^{\varphi_N}(y_i) ) \lambda_p c_i , \quad x\in \widetilde{M}.$$
\end{Lem}
\begin{proof}
	Recall from the proof of Lemma \ref{operator valued weight} that the canonical conditional expectation from $(\widetilde{M}, \widehat{\varphi})$ to $(\widetilde{B},\widehat{\varphi}_B)$ is given by $E_{B\rtimes \R}((x\otimes b)\lambda_t)=\varphi_N(x) b \lambda_t$ for $x\in N$, $b\in B$ and $t\in \R$. 
For $x\in \widetilde{M}$, we calculate that
\begin{eqnarray*}
	\widetilde{\psi}\circ \omega (x) 
	&=& \widetilde{\psi}(\lambda_q^*E_M(z^* x y)\lambda_p)\\
	&=& \sum_{i=1}^n  (\varphi_N(z_i^* \, \cdot \, y_i)\otimes \id_B \otimes \id_{L^2(\R)} ) (\lambda_q^*E_M(z^* x y)\lambda_p) c_i \\
	&=& \sum_{i=1}^n  E_{B\rtimes \R} (z_i^* \lambda_q^*E_M(z^* x y)\lambda_p y_i) c_i \\
	&=& \sum_{i=1}^n \lambda_q^* E_{B\rtimes \R} \circ E_M(\sigma_q^{\varphi_N}(z_i^*) z^* x y \sigma_p^{\varphi_N}(y_i) ) \lambda_p c_i .
\end{eqnarray*}
Since $E_{B\rtimes \R} \circ E_M = E_B$ by Lemma \ref{operator valued weight}, we obtain the conclusion.
\end{proof}

\begin{Lem}\label{key approximation lemma}
	Suppose that $N$ has the $\varphi_N$-W$^*$CMAP. Then there exists a net $(\varphi_\lambda)_\lambda$ of c.c.\ maps on $\widetilde{M}$ such that $\varphi_\lambda\to \id_{\widetilde{M}}$ point $\sigma$-weakly and that each $\varphi_\lambda$ is a finite sum of $d^*E_B(z^* \, \cdot \,y)c$ for some $c,d\in \widetilde{M}$ and $y,z\in \mathfrak{n}_{E_B}$.
\end{Lem}
\begin{proof}
	Fix a net $(\psi_i)_i$ of normal c.c.\ maps on $N$ as in Definition \ref{WCMAP with respect to a state} and put $(\widetilde{\psi}_i)_i$ as in the statement of the previous lemma. Let $(\omega_j)_j$ be a net of c.c.p.\ maps on $\widetilde{M}$ given by Lemma \ref{key lemma for approximation1}. Then by Lemma \ref{key lemma for approximation2} the composition $\widetilde{\psi}_i\circ \omega_j$ is a finite sum of $d^*E_B(z^* \, \cdot \,y)c$ for some $c,d\in \widetilde{M}$ and $y,z\in \mathfrak{n}_{E_B}$. Since $\lim_i(\lim_j \widetilde{\psi}_i\circ \omega_j)=\id_{\widetilde{M}}$ in the point $\sigma$-weak topology, it is easy to show that for any finite subset $\mathcal{F}\subset \widetilde{M}$ and any $\sigma$-weak neighborhood $\mathcal{V}$ of $0$, there is $i$ and $j$ such that $\widetilde{\psi}_i\circ \omega_j(x)-x\in \mathcal{V}$ for all $x\in \mathcal{F}$. 
So putting this $\widetilde{\psi}_i\circ \omega_j$ as $\varphi_{(\mathcal{F},\mathcal{V})}$, one can construct a desired net $(\varphi_\lambda)_\lambda:=(\varphi_{(\mathcal{F},\mathcal{V})})_{(\mathcal{F},\mathcal{V})}$. 
\end{proof}

\subsection{\bf Relative weakly compact actions on continuous cores}\label{Relative weakly compact actions on continuous cores}

	We keep the notation from the previous subsection, such as $M=N\ovt B$ and $\varphi=\varphi_N\otimes \varphi_B$. Let $\Tr$ be an arbitrary semifinite trace on $\widetilde{M}$, $p\in \widetilde{M}$ a projection with $\Tr(p)
=1$, and $A\subset p\widetilde{M}p$ a von Neumann subalgebra with expectation $E_A$. In this subsection, we prove that under some assumptions on $A$ and $M$, the normalizer of $A$ in $pMp$ acts on $A$ as a weakly compact action with an appropriate representation. 

	Since our proof is a generalization of the one of \cite[Theorem 5.1]{PV11}, we make use of the following notation, which are similar to the ones used in \cite[Theorem 5.1]{PV11}:
	\begin{eqnarray*}
		&&H := L^2(\widetilde{M}, \widehat{\varphi})\otimes_{B} L^2(\widetilde{M}, \Tr), \text{ with left, right actions }\pi_H, \theta_H, \\
		&&\mathcal{M}_{H}:=\mathrm{W}^*\{ \pi_H(\widetilde{M}), \theta_H(\widetilde{M}^\op )\} \subset \B(H), \\
		&&\mathcal{H} := (\theta_H(p)H) \otimes_A pL^2(\widetilde{M}, \Tr),\\
		&&\pi_\mathcal{H} \colon \widetilde{M} \ni x \mapsto (x\otimes_B p^\op )  \otimes_A p \in \mathbb{B}(\mathcal{H}),\\
		&&\theta_\mathcal{H} \colon \widetilde{M}^\op  \ni y^\op \mapsto (1\otimes_B p^\op ) \otimes_A y^\op  \in \mathbb{B}(\mathcal{H}),\\
		&&\mathcal{M}:=\mathrm{W}^*\{\pi_\mathcal{H}(\widetilde{M}), \theta_\mathcal{H}(\widetilde{M}^\op )\}\subset \B(\mathcal{H}).
	\end{eqnarray*}

	As  we observed in Lemma \ref{lem relative weakly compact action}, we actually use the weakly compact action with the standard representation of $\mathcal{M}$. So we first observe that $\mathcal{M}$ admits a useful identification as a crossed product, and so its standard representation is taken as a simple form.

\begin{Lem}\label{standard representation}
	Let $X\subset \mathcal{M}$ be the von Neumann subalgebra generated by $\pi_{\mathcal{H}}(B)$ and $\theta_{\mathcal{H}}(\widetilde{M}^\op )$, and let $X \subset \B(L^2(X))$ be a standard representation, so that $B$ and $\widetilde{M}^\op $ acts on $L^2(X)$. 
Then $\mathcal{M}$ is isomorphic to the crossed product von Neumann algebra $\R \ltimes (N\ovt X)$ by the diagonal action $\sigma^{\varphi_N}\otimes \alpha^X$, where $\alpha^X$ is given by $\alpha_t^X(\pi_\mathcal{H}(b) \theta_{\mathcal{H}}(y^\op ) )=\pi_\mathcal{H}(\sigma_t^{\varphi_B}(b)) \theta_{\mathcal{H}}(y^\op ) $ for $t\in \R$, $b\in B$, and $y\in \widetilde{M}$.

In particular the standard representation of $\mathcal{M}$ is given by $L^2(\R) \otimes L^2(N) \otimes L^2(X)$ with the following representation: 
For any $\xi \in L^2(\R) \otimes L^2(N) \otimes L^2(X)= L^2(\R , L^2(N)\otimes L^2(X))$ and $s\in \R$,
\begin{itemize}
	\item $L\R \ni \lambda_t \mapsto \lambda_t \otimes 1_N \otimes 1_X$; $((\lambda_t \otimes 1_N \otimes 1_X)\xi)(s) := \xi(s-t)$;
	\item $N \ni x \mapsto \pi_{\sigma^{\varphi^N}}(x) \otimes 1_X$; $((\pi_{\sigma^{\varphi^N}}(x) \otimes 1_X)\xi)(s):=(\sigma_{-s}^{\varphi_N}(x) \otimes 1_X) \xi (s)$;
	\item $B \ni b \mapsto \pi_{\sigma^{\varphi^B}}(b)_{\rm 13}$; $((\pi_{\sigma^{\varphi^B}}(b)_{\rm 13})\xi)(s):=(1_N\otimes \sigma_{-s}^{\varphi_B}(b)) \xi (s)$;
	\item $\widetilde{M}^\op  \ni y^\op  \mapsto 1_{L^2(\R)} \otimes 1_N \otimes y^\op $; $((1_\R \otimes 1_N \otimes y^\op )\xi)(s):=(1_N\otimes y^\op  )\xi(s)$.
\end{itemize}
\end{Lem}
\begin{proof}
	By Proposition \ref{lemma for relative tensor}, $H$ is isomorphic to $L^2(\R)\otimes L^2(N) \otimes L^2(B) \otimes L^2(N)\otimes L^2(\R) $. Since the right $\widetilde{M}$-action acts only on the right three Hilbert spaces, the Hilbert space $\mathcal{H}=H\otimes_A pL^2(\widetilde{M}, \Tr)$ is identified as $L^2(\R)\otimes L^2(N) \otimes K$, where 
	$$K:= \theta_H(p^\op )(L^2(B) \otimes L^2(N) \otimes L^2(\R)) \otimes_A pL^2(\widetilde{M}, \Tr).$$ Note that $\widetilde{M}^\op $ acts on $K$ by $\theta_{\mathcal{H}}$, and $B$ acts on $L^2(\R) \otimes K$ by $\pi_{\mathcal{H}}$, so that $X$ acts on $L^2(\R) \otimes K$. More precisely we have $X \subset L^\infty(\R)\ovt \C 1_N \ovt \B(K)$. 

	Let $W$ be a unitary on $L^2(\R)\otimes L^2(N)$ given by $(W\xi)(t):= \Delta^{it}_{\varphi_N} \xi(t)$ for $t\in \R$ and $\xi \in L^2(\R)\otimes L^2(N)=L^2(\R, L^2(N))$. It satisfies that for any $f\in L^\infty(\R)$, $t\in \R$, and $x\in N$,
	$$W \pi_{\sigma^{\varphi_N}}(x) W^* = 1_{L^2(\R)} \otimes x, \quad W (\lambda_t \otimes 1_N )W^*= \lambda_t \otimes \Delta^{it}_{\varphi_N}, \quad \text{and} \quad W(f\otimes 1_N)W^*=f\otimes 1_N.$$
Let next $V$ be a unitary on $L^2(\R)\otimes L^2(\R)$ defined similarly to $W$ exchanging $\Delta_{\varphi_N}^{it}$ with $\lambda_t$, so that it satisfies for $t\in \R$ and $f\in L^\infty(\R)$,
	$$V (1\otimes \lambda_t )V^* = \lambda_t \otimes \lambda_t \quad \text{and} \quad V(1\otimes f)V^* = 1\otimes f .$$ 
Define then a unitary on $L^2(\R)\otimes \mathcal{H}$ by $U:= (V \otimes 1_N \otimes 1_K) (1_{L^2(\R)}\otimes W \otimes 1_K)$. One can show that $\Ad U = \id$ on $\C 1_{L^2(\R) }\otimes X \subset \C1_{L^2(\R)} \ovt L^\infty(\R)\ovt \C 1_N \ovt \B(K)$, and
\begin{itemize}
	\item $\Ad U (1_{L^2(\R)} \otimes \lambda_t \otimes 1_N \otimes 1_K )  =  (\lambda_t \otimes \lambda_t \otimes \Delta_{\varphi_N}^{it} \otimes 1_K ) $, \quad for $t\in \R$,
	\item $\Ad U (1_{L^2(\R)} \otimes \pi_{\sigma^{\varphi_N}}(x) \otimes 1_K)  =  (1_{L^2(\R)} \otimes 1_{L^2(\R)} \otimes x \otimes 1_K ) $, \quad for $x\in N$.
\end{itemize}
Then $\Ad U (\mathcal{M})$ is identified as the crossed product von Neumann algebra $\R \ltimes (N\ovt X)$ given by the $\R$-action $\sigma^{\varphi_N} \otimes \alpha^X$, where $\alpha^X$ is given by $\Ad (\lambda_t \otimes 1_N \otimes 1_K)$ using $X \subset L^\infty(\R)\otimes \C 1_N \otimes \B(K)$, which is exactly the action given in the statement. 
Finally one can choose the standard representation of $\R \ltimes (N\ovt X)$ as in the statement and we can end the proof.
\end{proof}

	Now we prove the main observation of this section. This is a generalization of \cite[Theorem 3.5]{OP07} and \cite[Theorem 5.1]{PV11}. Since we already obtained approximation maps for $\widetilde{M}$ in Lemma \ref{key approximation lemma}, which are ``relative to $B$'', almost same arguments as of \cite[Theorem 3.5]{OP07} and \cite[Theorem 5.1]{PV11} work. However, since our approximation maps are not defined directly on $\mathcal{M}_{H}$, we need a stronger assumption on the subalgebra $A$, namely, we need \textit{amenability}, instead of relative amenability. See Step 1 in the proof below and observe that we really need amenability for a subalgebra $Q\subset pMp$.

\begin{Thm}\label{weakly compact action on core}
	Keep the setting above and suppose the following conditions. 
	\begin{itemize}
		\item The algebra $B$ is a type $\rm III_1$ factor.
		\item The algebra $A$ is amenable. 
		\item The algebra $N$ has the $\varphi_N$-W$^*$CMAP.
	\end{itemize}
Then $\mathcal{N}_{p\widetilde{M}p}(A)$ acts on $A$ as a weakly compact action for $(\widetilde{M}, \Tr, \pi_{\mathcal{H}}, \theta_{\mathcal{H}}, \mathcal{M})$.
\end{Thm}
\begin{proof}
	The proof consists of several steps.  
For any von Neumann subalgebra $Q\subset p\widetilde{M}p$, we write as $\mathcal{C}_{H,Q}$ (resp.\ $\mathcal{M}_{H, Q}$) the C$^*$-algebra (resp.\ the von Neumann algebra) generated by $\pi_H(p\widetilde{M}p)\theta_H(Q^\op )$.

\begin{step}
{\bf
	Using the $\varphi_N$-W$^*$CMAP of $N$, we construct a net of normal functionals on $\mathcal{M}_{H}$ which are contractive on $\mathcal{M}_{H,Q}$ for any amenable $Q$.
}
\end{step}

	In this step, we show that there is a net $(\mu_i)_i$ of normal functional on $\mathcal{M}_H$ such that 
	\begin{itemize}
		\item $\mu_i( \pi_H( a )  \theta_H(b^\op )) = \Tr(p\varphi_i(a)pbp)$ for all $a,b\in \widetilde{M}$,
		\item we have $\|\mu_i |_{\mathcal{M}_{H, Q}} \|\leq 1$ for any amenable von Neumann subalgebra $Q\subset p\widetilde{M}p$. 
	\end{itemize}
By Lemma \ref{key approximation lemma}, there exists a net $(\varphi_i)_i$ of c.c.\ maps on $\widetilde{M}$ such that $\varphi_i\to \id_{\widetilde{M}}$ point $\sigma$-weakly and that each $\varphi_i$ is a finite sum of $d^*E_B(z^* \, \cdot \,y)c$ for $c,d\in \widetilde{M}$ and $y,z\in \mathfrak{n}_{E_B}$. Observe that for any functional $d^*E_B(z^* \, \cdot \,y)c$ for some $c,d\in \widetilde{M}$ and $y,z\in \mathfrak{n}_{E_B}$, one can define an associated normal functional on $\mathcal{M}_H$ by 
	$$\mathcal{M}_H \ni T \mapsto \langle T (\Lambda_{\widehat{\varphi}}(y)\otimes_B \Lambda_{\Tr}(cp)), \Lambda_{\widehat{\varphi}}(z)\otimes_B \Lambda_{\Tr}(dp) \rangle_H.$$
In this way, since $\varphi_i$ is a finite sum of such maps, one can associate each $\varphi_i$ with a normal functional on $\mathcal{M}_H$, which we denote by $\mu_i$. 
Then by the formula $L_{\Lambda_{\widehat{\varphi}}(z)}^* a L_{\Lambda_{\widehat{\varphi}}(y)} = E_B(z^*ay)$ for $x,y\in \mathfrak{n}_{E_B}\cap \mathfrak{n}_\varphi$ and $a\in \widetilde{M}$, it is easy to verify that $\mu_i( \pi_H( a ) \theta_H(b^\op )) = \Tr(p\varphi_i(a)pbp)$ for $a,b\in \widetilde{M}$. We need to show that $\|\mu_i |_{\mathcal{M}_{H, Q}} \|\leq 1$ for any amenable $Q\subset p\widetilde{M}p$. For this, since $\mu_i$ is normal, we have only to show that $\|\mu_i |_{\mathcal{C}_{H, Q}} \|\leq 1$.

	By Lemma \ref{key lemma for approximation3} below, since $B$ is a type III$_1$ factor, the $\ast$-algebra generated by $\pi_H(\widetilde{M})$ and $\theta_H(\widetilde{M}^\op )$ is isomorphic to $\widetilde{M} \ota \widetilde{M}^\op $. So for any amenable $Q\subset p\widetilde{M}p$,  the C$^*$-algebra generated by $\pi_H(\widetilde{M})\theta_H(Q^\op )$ is isomorphic to $\widetilde{M}\otm Q^\op $. Hence one can define c.c.\ maps $\varphi_i \otimes \id_{Q^\op }$ on $\mathcal{C}_{H,Q}$. Since $Q$ is amenable, one has
	$$ {}_{\widetilde{M}}L^2(\widetilde{M}p)_{Q} \prec {}_{\widetilde{M}} (\theta_H(p^\op ) H)_{Q} .$$
Finally if we write as $\nu$ the associated $\ast$-homomorphism as with this weak containment, then the functional $T \mapsto \langle \nu\circ (\varphi_i\otimes \id_{Q^\op })(T) \Lambda_{\Tr}(p), \Lambda_{\Tr}(p)\rangle_{\Tr}$ coincides with $\mu_i$ on $\mathcal{C}_{H,Q}$, and hence we obtain $\| \mu_i |_{\mathcal{C}_{H,Q}}\|\leq 1$. Thus we obtained a desired net $(\mu_i)_i$.

\begin{step}
{\bf
	Using the amenability of $A$, the absolute values of normal functionals $(\mu_i)_i$ constructed in Step 1 satisfies desired properties on $\mathcal{M}_{H,A}$.
}
\end{step}

	Before this step, recall from the first part of the proof of \cite[Theorem 3.5]{OP07} that for any C$^*$-algebra $C$, any state $\omega$ on $C$ and any partial isometry $u\in C$ with $p:=uu^*$ and $q:=u^*u$, one has 
	$$\mathrm{max} \{\|\omega(\, \cdot \, u^*) - \omega(\, \cdot \, q)\|^2, \|\omega(u \, \cdot \, u^*) - \omega(q\, \cdot \, q)\|^2 \} \leq 4  \left(\omega(p)+\omega(q)-\omega(u)-\omega(u^*)\right).$$

	Let $(\mu_i)_i$ be a net constructed in Step 1. For notation simplicity, for any amenable von Neumann subalgebra $Q\subset p\widetilde{M}p$ we denote by $\mu_i^Q$ the restriction of $\mu_i$ on $\mathcal{M}_{H,Q}$. 
\begin{claim}
	For any amenable $Q$, one has
	$$\| \mu_i^Q \|\to 1 \quad \text{and} \quad \|\mu_i^Q - |\mu_i^Q| \| \to 0,$$
where $|\mu_i^Q|$ is the absolute value of $\mu_i^Q$.
\end{claim}
\begin{proof}[Proof of Claim]
	By Step 1, we know $\|\mu_i^Q \|\leq 1$ and hence $\| \mu_i^Q \|\to 1$, since $\mu_i(\pi_H(p) \theta_H(p^\op )) \to 1$. 
Let $\mu_i^Q = |\mu_i^Q|(\, \cdot \, u_i)$ be the polar decomposition with a partial isometry $u_i\in \mathcal{M}_{H,Q}$. 
For $p_i:=u_iu_i^*$ and $q_i:=u_i^*u_i$, it holds that
	$$|\mu_i^Q| = \mu_i^Q(\, \cdot \, u_i^*), \quad |\mu_i^Q| = |\mu_i^Q|(q_i\, \cdot \, q_i), \quad\text{and} \quad \mu_i^Q = \mu_i^Q(\, \cdot \, p_i) = \mu_i^Q(q_i\, \cdot \, ).$$ 
The final equation says that $\mu_i^Q(p_i) = \mu_i^Q(1_Q)\to 1$. Then by the inequality at the beginning of this step, we have
\begin{eqnarray*}
	\|\mu_i^Q - |\mu_i^Q| \|^2 
	&=& \||\mu_i^Q|(\, \cdot \, u_i^*) - |\mu_i^Q|(\, \cdot \, q_i )\|^2 \\
	&\leq& 4 \, \left( |\mu_i^Q|(p_i)+|\mu_i^Q|(q_i)-|\mu_i^Q|(u_i)-|\mu_i^Q|(u_i^*) \right) \\
	&\leq& 4 \, \left( \|\mu_i^Q\| + \|\mu_i^Q\| - 2\mathrm{Re}(\mu_i^Q(p_i) ) \right)  \to 0.
\end{eqnarray*}
This completes the claim.
\end{proof}

	Put $\omega_i:= |\mu_i^A|/ \|\mu_i^A \|$. In this step, we show that $(\omega_i)_i$ satisfies the following conditions:
\begin{itemize}
	\item[$\rm (i)$] $\omega_i(\pi_H(x) \theta_H(p^\op ))\to \Tr(pxp)$, \quad for all $x\in p\widetilde{M}p$;
	\item[$\rm (ii)$] $\omega_i(\pi_H(a)\theta_H(\bar{a}))\to 1$, \quad for all $a \in \mathcal{U}(A)$;
	\item[$\rm (iii)$] $\|\omega_i \circ \Ad (\pi_H(u)\theta_H(\bar{u})) - \omega_i \|_{\mathcal{M}_{H,A}^*} \to 0$, \quad for all $u \in \mathcal{N}_{p\widetilde{M}p}(A)$.
\end{itemize}
Since $\|\mu_i^A\| \to 1$ and $\|\mu_i^A - |\mu_i^A| \| \to 0$, to verify these three conditions, we have only to show that $(\mu_i)_i$ satisfies the same conditions. 
Then by construction, it is easy to verify (i) and (ii). So we will check only the final condition.

	Fix $u \in \mathcal{N}_{p\widetilde{M}p}(A)$ and recall that the von Neumann algebra $A^u$ generated by $A$ and $u$ is amenable \cite[Lemma 3.4]{OP07}. 
Hence by Step 1, $\| |\mu_i^{A^u}| - \mu_i^{A^u} \|_{\mathcal{M}_{H,A^u}^*} \to 0$. Combined with the inequality at the beginning of this step, putting $U:=\pi_H(u) \theta_H(\bar{u})$, we have
\begin{eqnarray*}
	\lim_i\|\mu_i^{A} \circ \Ad U - \mu_i^{A}\|^2_{\mathcal{M}_{H,A}^*}
	&\leq& \lim_i\|\mu_i^{A^u} \circ \Ad U - \mu_i^{A^u}\|^2_{\mathcal{M}_{H,A^u}^*}\\
	&=& \lim_i\||\mu_i^{A^u}| \circ \Ad U - |\mu_i^{A^u}|\|^2_{\mathcal{M}_{H,A^u}^*} \\
	&\leq& \lim_i 4\, \left( 2 - 2\mathrm{Re}(|\mu_i^{A^u}| (U)) \right) \\
	&=& \lim_i 4\, \left( 2 - 2\mathrm{Re}(\mu_i^{A^u} (U)) \right) = 0 .
\end{eqnarray*}
Thus we proved that the net $(\omega_i)_i$ of normal states on $\mathcal{M}_{H}$ satisfies conditions (i), (ii) and (iii) above.

\begin{step}
{\bf
	Using a normal u.c.p.\ map from $\mathcal{M}$ to $\mathcal{M}_{H,A}$, we obtain desired functionals on $\mathcal{M}$.
}
\end{step}

	In this step, we first construct a normal u.c.p.\ map $\mathcal{E} \colon \mathcal{M} \to \mathcal{M}_{H,A}$ satisfying 
	$$ \mathcal{E}( \pi_{\mathcal{H}}(a) \theta_{\mathcal{H}} (b^\op ) ) = \pi_H(pap) \theta_H(E_A(pbp)^\op ), \quad \text{for any } a,b \in \widetilde{M},$$
where $E_A$ is the unique $\Tr$-preserving conditional expectation from $p\widetilde{M}p$ onto $A$.

	For this, observe first that for any right $A$-module $K$ with the right action $\theta_K$, there is an isometry $V_K \colon K \to K\otimes_A pL^2(\widetilde{M},\Tr)$ given by $V\xi = \xi\otimes_A \Lambda_{\Tr}(p)$ for any left $\Tr$-bounded vector $\xi\in K$. Indeed, using the fact $\Lambda_{\Tr}(p)= J_{\Tr}\Lambda_{\Tr}(p)$, one has 
	$$\| V \xi \|= \|\xi \otimes_A \Lambda_{\Tr}(p)\|=\|L_\xi \Lambda_{\Tr}(p)\|_{2,\Tr}=\|L_\xi \Lambda_{\Tr}(p)\|_{2,\Tr} = \|\theta_K(p^\op )\xi\|_K=\|\xi\|_K.$$
Hence, since $\pi_H(p)\theta_H(p^\op ) H$ is a right $A$-module, one can define an isometry 
	$$V\colon \pi_H(p)\theta_H(p^\op ) H \to \pi_{\mathcal{H}}(p)\theta_{\mathcal{H}}(p^\op )\mathcal{H}\subset \mathcal{H}; \quad V \xi := \xi \otimes_A \Lambda_{\Tr}(p) .$$ 
It is then easy to verify that 
	$$V^* \pi_{\mathcal{H}}(a) \theta_{\mathcal{H}} (b^\op ) V = \pi_H(pap) \theta_H(E_A(pbp)^\op ), \quad \text{for any } a,b \in \widetilde{M}.$$
Thus we obtain a normal u.c.p.\ map $\mathcal{E} \colon \mathcal{M} \to \mathcal{M}_{H,A}$ by $\mathcal{E}(T) := V^* T V$. 

	Let now $(\omega_i)_i$ be the net of normal states on $\mathcal{M}_{H,A}$ constructed in Step 2. By conditions (i) and (ii) on $(\omega_i)_i$, it is easy to see that normal states $\gamma_i:=\omega_i\circ \mathcal{E}$ on $\mathcal{M}$ satisfy 
\begin{itemize}
	\item[$\rm (i)'$] $\gamma_i(\pi_{\mathcal{H}}(x))\to \tau(pxp)$, \quad for all $x\in \widetilde{M}$;
	\item[$\rm (ii)'$] $\gamma_i(\pi_{\mathcal{H}}(a)\theta_{\mathcal{H}}(\bar{a}))\to 1$, \quad for all $a \in \mathcal{U}(A)$.
\end{itemize}
Finally since $E_A$ satisfies $E_A \circ \Ad u = \Ad u \circ E_A $ for any $u\in \mathcal{N}_{p\widetilde{M}p}(A)$, one has 
	$$\gamma_i\circ \Ad (\pi_{\mathcal{H}}(u)\theta_{\mathcal{H}}(\bar{u})) = \omega_i \circ \Ad (\pi_{\mathcal{H}}(u)\theta_{\mathcal{H}}(\bar{u})) \circ \mathcal{E}$$
on $\pi_{\mathcal{H}}(\widetilde{M})\theta_{\mathcal{H}}(\widetilde{M})$, and hence on $\mathcal{M}$ by normality. So condition (iii) on $(\omega_i)_i$ shows
\begin{itemize}
	\item[$\rm (iii)'$] $\|\gamma_i \circ \Ad (\pi_{\mathcal{H}}(u)\theta_{\mathcal{H}}(\bar{u})) - \gamma_i \| \to 0$, \quad for all $u \in \mathcal{N}_{p\widetilde{M}p}(A)$.
\end{itemize}
Thus the net $(\gamma_i)_i$ on $\mathcal{M}$ satisfies conditions $\rm (i)'$, $\rm (ii)'$ and $\rm (iii)'$. By Proposition \ref{lem relative weakly compact action}(2), we conclude that $\mathcal{N}_{p\widetilde{M}p}(A)$ acts on $A$ weakly compactly for $(\widetilde{M}, \Tr, \pi_{\mathcal{H}}, \theta_{\mathcal{H}}, \mathcal{M})$.
\end{proof}

We prove a lemma used in the proof above.

\begin{Lem}\label{key lemma for approximation3}
	Assume that $B$ is a type $\rm III_1$ factor. Then the $\ast$-algebra generated by $\pi_H(\widetilde{M})$ and $\theta_H(\widetilde{M}^\op )$ is isomorphic to $\widetilde{M}\ota \widetilde{M}^\op $.
\end{Lem}
\begin{proof}
	Let $\nu \colon \widetilde{M}\ota \widetilde{M}^\op  \to \ast\textrm{-alg}\{ \pi_H(\widetilde{M}), \theta_H(\widetilde{M}^\op ) \}$ be a $\ast$-homomorphism given by $\nu(x \otimes y^\op ) = \pi_H(x) \theta_H(y^\op )$ for $x,y \in \widetilde{M}$. We will show that $\nu$ is injective. 

	Assume that $\nu (\sum_{i=1}^n x_i \otimes y_i^\op ) = \sum_{i=1}^n \pi_H(x_i)\theta_H(y_i^\op )=0$ for some $x_i,y_i \in \widetilde{M}$. We may assume $y_i\neq 0$ for all $i$. Put
	\begin{eqnarray*}
X:=\left[ 
\begin{array}{cccc}
\pi_H(x_1) & \pi_H(x_2) & \cdots & \pi_H(x_n) \\
0 & 0 & \cdots & 0 \\
\vdots & \vdots & \ddots & \vdots \\
0 & 0 & \cdots & 0 \\
\end{array} 
\right]
\quad \text{and} \quad 
Y:=\left[ 
\begin{array}{cccc}
\theta_H(y_1^\op ) & 0 & \cdots & 0 \\
\theta_H(y_2^\op ) & 0 & \cdots & 0 \\
\vdots & \vdots & \ddots & \vdots \\
\theta_H(y_n^\op ) & 0 & \cdots & 0 \\
\end{array} 
\right]
	\end{eqnarray*}
and observe $XY=0$. We regard them as elements in $\B(H)\otimes \mathbb{M}_n$. Let $p$ be the left support projection of $Y$ which is contained in $\theta_H(\widetilde{M}^\op ) \otimes \M_n$ and satisfies $Xp=0$. Since $Xupu^*=0$ for any unitary $u\in \B(H)\otimes \M_n$ which commutes with $X$, and since $\theta_H(\widetilde{M}^\op )\otimes \C^n $ commutes with $X$ (where $\C^n \subset \M_n$ is the diagonal embedding), we have $X z =0$ for $z:= \sup\{ upu^* \mid u \in \mathcal{U}(\theta_H(\widetilde{M}^\op )\otimes \C^n)\}$. Observe that $z$ is contained in 
	$$(\theta_H(\widetilde{M}^\op )\otimes \M_n) \cap (\theta_H(\widetilde{M}^\op )\otimes \C^n)' = \theta_H(\mathcal{Z}(\widetilde{M})^\op )\otimes \C^n$$
and hence we can write $z= (z_i)_{i=1}^n$ for some $z_i \in \theta_H(\mathcal{Z}(\widetilde{M})^\op )$. Then the condition $Xz=0$ is equivalent to $\pi_H(x_i)z_i=0$ for all $i$. Observe also that $z_i\neq 0$ for all $i$. Indeed, since $z \geq p$ and $pY=Y$, we have $zY=Y$ and hence $z_i \theta_H(y_i^\op )=\theta_H(y_i^\op )$. This implies $z_i\neq 0$ since we assume $y_i\neq 0$ for all $i$. 

	Now we claim that $\pi_H(x_i)z_i=0$ is equivalent to $x_i=0$ or $z_i =0$. Once we prove the claim, since $z_i\neq 0$, we have $x_i=0$ and so $\sum_{i=1}^nx_i\otimes y_i^\op =0$, that means injectivity of $\nu$.

	By Lemma \ref{center of core of III1}, the center of $\widetilde{M}$ coincides with $\mathcal{Z}(N)$. Then by Proposition \ref{lemma for relative tensor}, we identify $H= L^2(\R)\otimes L^2(N)\otimes L^2(B,\psi_B) \otimes L^2(N)\otimes L^2(\R)$ on which we have 
	\begin{align*}
	&\pi_H(\widetilde{M}) \subset \B(L^2(\R)\otimes L^2(N)\otimes L^2(B,\psi_B)) \otimes \C1_{L^2(N)\otimes L^2(\R)} , \\
	&\theta_H(\widetilde{M}^\op ) \subset \C1_{L^2(\R)\otimes L^2(N)} \otimes \B(L^2(B,\psi_B) \otimes L^2(N)\otimes L^2(\R)).
	\end{align*}
In particular $\theta_H(\mathcal{Z}(\widetilde{M})^\op ) = \theta_H(\mathcal{Z}(N)) \subset \C1_{L^2(\R)\otimes L^2(N) \otimes L^2(B,\psi_B)} \otimes \B(L^2(N)\otimes L^2(\R))$, and hence the C$^*$-algebra generated by $\pi_H(\widetilde{M})$ and $\theta_H(\mathcal{Z}(\widetilde{M})^\op )$ is isomorphic to $\widetilde{M} \otm \mathcal{Z}(\widetilde{M})^\op $. Thus since $z_i\in \theta_H(\mathcal{Z}(\widetilde{M})^\op )$, the condition $\pi_H(x_i)z_i=0$ is equivalent to $x_i=0$ or $z_i=0$. This completes the proof.
\end{proof}

\section{\bf Proof of Theorem \ref{thmA}}\label{Proof of Theorem A free quantum group factors}

	In this section, we prove Theorem \ref{thmA}. We follow the proof of \cite[Theorem B]{Is13}, which originally comes from the one of \cite[Theorem 1.4]{PV12}.

\subsection{\bf Some general lemmas}\label{Some general properties}

	Let $\G$ be a compact quantum group with the Haar state $h$ and put $N_0:=C_{\rm red}(\G)\subset L^\infty(G) =:N$ and $\varphi_N:=h$. Let $(X,\varphi_X)$ be a von Neumann algebra with a faithful normal semifinite weight. Let $\alpha^X$ be an action of $\R$ on $X$ and put $\alpha:=\sigma^{\varphi_N}\otimes \alpha^X$ and $\mathcal{M}:=(N\ovt X)\rtimes_{\alpha}\R$.

	In this setting, we prove two general lemmas. We use the following general fact for quantum groups.
\begin{itemize}
	\item For any $x\in \mathrm{Irred}(\G)$, there is an orthonormal basis $\{u_{i,j}^x \}_{i,j} \subset C_{\rm red}(\G)$ of $H_x$ with $\lambda_{i,j}^x>0$ such that $\sigma_t^h(u_{i,j}^x) = \lambda_{i,j}^x u_{i,j}^x$ for all $t\in \R$. 
\end{itemize}
Recall that all the linear spans of such a basis, which is usually called a dense Hopf $\ast$-algebra, make a norm dense $\ast$-subalgebra of $C_{\rm red}(\G)$. We note that each matrix $(u_{i,j}^x)_{i,j}$ may not be a unitary, since we assume $\{u_{i,j}^x \}_{i,j}$ is orthonormal (i.e.\ they are normalized). 
\begin{conv}
	Throughout this section, we fix such a basis $\{u_{i,j}^x \}_{i,j}^x$. For notation simplicity, we identify any subset $\mathcal{E}\subset \mathrm{Irred}(\G)$ (possibly $\mathcal{E}=\mathrm{Irred}(\G)$) with the set $\{u_{i,j}^x \mid x\in\mathcal{E}, i,j\}$. 
\end{conv}
Note that this identification will not make any confusion, since in proofs of this section we only use the property that $\mathcal{E}\subset \mathrm{Irred}(\G)$ is a finite set.

	Here we record an elementary lemma.

\begin{Lem}\label{lemma for free quantum group1}
	For any $a\in N_0$, the element $\pi_{\sigma^{\varphi_N}}(a)\in N\rtimes_{\sigma^{\varphi_N}} \R \subset \B(L^2(N)\otimes L^2(\R))$ is contained in $N_0\otm C_b (\R)$, where $C_b(\R)$ is the set of all norm continuous bounded functions on $\R$. 
\end{Lem}
\begin{proof}
	We may assume that $a$ is an eigenvector, namely, $\sigma_t^{\varphi_N}(a)=\lambda^{it} a$ for some $\lambda>0$. Then since $(\pi_{\sigma^{\varphi_N}}(a)\xi)(t) = \sigma_{-t}^{\varphi_N}(a)\xi(t)=\lambda^{-it} a\xi(t)$ for $t\in \R$, one has $\pi_{\sigma^{\varphi_N}}(a) = a \otimes f$, where $f\in C_{b}(\R)$ is given by $f(t):=\lambda^{-it}$. Hence we get the conclusion.
\end{proof}

	We fix a faithful normal semifinite weight $\varphi_X$ on $X$ and put $\psi:=\varphi_N \otimes \varphi_X$ with its dual weight $\widehat{\psi}$. Recall that the compression map by $ P_N \otimes 1_X \otimes 1_{L^2(\R)}$, where $P_N$ is the one dimensional projection from $L^2(N)$ onto $\C \Lambda_{\varphi_N}(1_N)$, is a conditional expectation $E_{X\rtimes \R}\colon \mathcal{M} \to X\rtimes \R$, which satisfies $\widehat{\psi}= \widehat{\varphi}_X\circ E_{X\rtimes \R}$ (this was shown in the first half of the proof of Lemma \ref{operator valued weight}). 
For any $a\in \mathcal{M}$ and $f\in C_c(\R, \mathcal{M})\mathfrak{n}_{\psi}$, we denote by $a f$ an element in $C_c(\R, \mathcal{M})\mathfrak{n}_{\psi}$ given by $t \mapsto \alpha_{-t}(a)f(t)$. Observe that $\Lambda_{\widehat{\psi}}(\widehat{\pi}_{\alpha}(af))= \pi_\alpha(a)\Lambda_{\widehat{\psi}}(\widehat{\pi}_{\alpha}(f))$. A simple computation shows that for any $a,b\in N$ and $f,g\in C_c(\R,X)\mathfrak{n}_{\varphi_X}$, 
\begin{eqnarray*}
	\langle af, bg \rangle_{\widehat{\psi}} 
	=  \langle a, b\rangle_{\varphi_N} \langle f,g \rangle_{\widehat{\varphi}_X} .
\end{eqnarray*}
Observe that all the linear spans of $uf$ for $u\in \mathrm{Irred}(\G)$ and $f\in C_c(\R,X)\mathfrak{n}_{\varphi_X}$ are dense in $L^2(N)\otimes L^2(X) \otimes L^2(\R)$. So if $\{f_\lambda\}_\lambda\subset C_c(\R,X)\mathfrak{n}_{\varphi_X}$ is an orthonormal basis in $L^2(X)\otimes L^2(\R)$, then the set $\{u f_\lambda\}_{u,\lambda}$ is an orthonormal basis of $L^2(N)\otimes L^2(X) \otimes L^2(\R)$. 
Along this basis, any $a \in \mathfrak{n}_{\widehat{\psi}}$ can be decomposed in $L^2(N)\otimes L^2(X) \otimes L^2(\R)$ as, for some $\alpha_{u,\lambda}\in \C$,
	$$\Lambda_{\widehat{\psi}}(a)= \sum_{u,\lambda} \alpha_{u,\lambda} u f_\lambda = \sum_{u,\lambda} \alpha_{u,\lambda}\pi_{\varphi_N}(u) \Lambda_{\widehat{\psi}}(\widehat{\pi}_{\alpha}(f_\lambda)) = \sum_{u} \pi_{\sigma^{\varphi_N}}(u) a_u$$ 
where $a_u = \sum_\lambda \alpha_{u,\lambda} f_\lambda \in L^2(\R, X)$. 
If we apply $(P_N \otimes 1_X \otimes 1_{L^2(\R)}) \pi_{\sigma^{\varphi_N}}(v^*)$ for some $v\in \mathrm{Irred}(\G)$ to this decomposition, then on the one hand
	$$(P_N \otimes 1_X \otimes 1_{L^2(\R)}) \pi_{\sigma^{\varphi_N}}(v^*) \, \Lambda_{\widehat{\psi}}(a)=(P_N \otimes 1_X \otimes 1_{L^2(\R)})  \, \Lambda_{\widehat{\psi}}(v^*a)=  \Lambda_{\widehat{\psi}}(E_{X\rtimes \R}(v^*a))$$
and on the other hand
\begin{eqnarray*}
	&&(P_N \otimes 1_X \otimes 1_{L^2(\R)}) \pi_{\sigma^{\varphi_N}}(v^*) \, \sum_{u} \pi_{\sigma^{\varphi_N}}(u) a_u \\
	&=& \sum_{u} \varphi_N(v^*u) a_u  
	=   \varphi_N(v^*v) a_v = a_v .
\end{eqnarray*}
Hence we have $a_v = \Lambda_{\widehat{\psi}}(E_{X\rtimes \R}(v^*a)) $ for all $v\in \mathrm{Irred}(\G)$. Thus we observed that any element $a\in\mathfrak{n}_{\widehat{\psi}}$ has the \textit{Fourier expansion} in the sense that 
	$$\Lambda_{\widehat{\psi}}(a)= \sum_{u} \pi_{\sigma^{\varphi_N}}(u)a_u = \sum_{u} \Lambda_{\widehat{\psi}}(uE_{X\rtimes \R}(u^*a)) , \quad \text{where } a_u =  \Lambda_{\widehat{\psi}}(E_{X\rtimes \R}(u^*a)) .$$
Using this property, we can prove the following lemma. We omit the proof, since it is straightforward.

\begin{Lem}\label{lemma for free quantum group2}
	Let $\mathcal{M}_0 \subset \mathcal{M}$ be the C$^*$-subalgebra generated by $N_0$ and $X\rtimes \R$. Then one has
\begin{eqnarray*}
	\mathcal{M}_0 &=& \overline{\mathrm{span}}^{\rm norm}\{ a x \mid a\in N_0, \ x\in X\rtimes \R\} \\
	&=& \overline{\mathrm{span}}^{\rm norm}\{ x a \mid a\in N_0, \ x\in X\rtimes \R\}.
\end{eqnarray*}
\end{Lem}

\subsection{\bf Proof of Theorem \ref{thmA}}\label{Proof for free quantum group factors}

	Let $\G$ be a compact quantum group with the Haar state $h$ and put $N_0:=C_{\rm red}(\G)\subset L^\infty(\G) =:N$ and $\varphi_N:=h$. Let $(B,\varphi_B)$ be a von Neumann algebra with a faithful normal state. 
We keep the notation from Subsections \ref{WCMAP with respect to a state produces approximation maps on continuous cores} and \ref{Relative weakly compact actions on continuous cores}, such as $M$,  $\varphi$, $\widetilde{B}$, $\widetilde{M}$, $\Tr$, $p$, $A$, $\mathcal{H}$, $\pi_{\mathcal{H}}$, $\theta_{\mathcal{H}}$, $\mathcal{M}$, except for the Hilbert space $H$ (which is used just below in a different method). Assume that $\Tr|_{\widetilde{B}}$ is semifinite. 
Recall that by Lemma \ref{standard representation}, $\mathcal{M}=\R \ltimes (N\ovt X)$ with the standard representation $L^2(\mathcal{M})=L^2(\R)\otimes L^2(N)\otimes L^2(X)$. 
Write as $\pi:=\pi_{\mathcal{H}}$ and $\theta:=\theta_{\mathcal{H}}$ for simplicity, and we sometimes omit $\pi$ and $\theta$ by regarding $\widetilde{M}, \widetilde{M}^\op  \subset \mathcal{M}$. Using Proposition \ref{lemma for relative tensor}, we put
\begin{align*}
	& H:=L^2(\mathcal{M})\otimes_{X} L^2(\mathcal{M})= L_\ell^2(\R)\otimes L_\ell^2(N) \otimes L^2(X) \otimes L_r^2(N) \otimes L_r^2(\R), \\ 
	& K:=L^2(\mathcal{M})\otimes_{(N\ovt X)} L^2(\mathcal{M})=L_\ell^2(\R)\otimes L^2(N) \otimes L^2(X)\otimes L_r^2(\R),
\end{align*}
and we denote by $\pi_H,$ $\rho_H$, $\pi_K$ and $\rho_K$ corresponding left and right actions of $\mathcal{M}$. Here we are using symbols $\ell$ and $r$ for $L^2(\R)$ and $L^2(N)$, so that $\pi_H$ and $\pi_K$ act on $L^2_\ell(\R)\otimes L_\ell^2(N)\otimes L^2(X)$ and $L^2_\ell(\R)\otimes L_\ell^2(N)\otimes L^2(X)$ respectively, and $\theta_H$ and $\theta_K$ act on $L^2(X)\otimes L_r^2(N)\otimes L^2_r(\R)$ and $L^2(N)\otimes L^2(X) \otimes L^2_r(\R)$ respectively. We denote by $\nu_{K,H}$ the corresponding $\ast$-homomorphism as $\mathcal{M}$-bimodules, which is \textit{not} bounded in general. 

	In this setting, we prove two lemmas. The first one uses bi-exactness of quantum groups, which corresponds to \cite[Lemma 4.1.3]{Is12}, while the second one uses Popa's intertwining techniques which corresponds to \cite[Lemma 4.1.2]{Is12}\cite[Lemma 4.4]{Is13}. See also \cite[Subsections 3.2 and 3.5]{PV12} for the origins of them. 

\begin{Lem}\label{lemma for free quantum group3}
	Assume that $\widehat{\mathbb{G}}$ is bi-exact with a u.c.p.\ map $\Theta$ as in the definition of bi-exactness. Let $\mathcal{M}_0$ be the C$^*$-algebra generated by $N_0$ and $\R \ltimes X$. Then $\Theta$ can be extended to a u.c.p.\ map 
	$$\widetilde{\Theta}\colon \mathrm{C}^*\{ \pi_H(\mathcal{M}_0), \theta_H(\mathcal{M}_0) \} \to \B(K) $$
which satisfies, using the flip $\Sigma_{12}\colon K\simeq L^2(N)\otimes L_\ell^2(\R) \otimes L^2(X)\otimes L_r^2(\R)$,
	$$\Sigma_{12}(\widetilde{\Theta}(\pi_H(xa) \theta_H(b^\op  y^\op )) - \pi_K(xa) \theta_K(b^\op  y^\op ))\Sigma_{12}\in  \K(L^2(N)) \otm \B(L_\ell^2(\R)\otimes L^2(X)\otimes L_r^2(\R)))$$
for any $a,b \in N_0$ and $x,y \in \R\ltimes X$.
\end{Lem}
\begin{proof}
	By applying flip maps, we identify:
\begin{align*}
	& H= L_\ell^2(N)\otimes L_r^2(N) \otimes L_\ell^2(\R) \otimes L^2(X)  \otimes L_r^2(\R), \\ 
	& K= L^2(N) \otimes L_\ell^2(\R)\otimes  L^2(X)\otimes L_r^2(\R).
\end{align*}
We define a u.c.p.\ map $\widetilde{\Theta}$ by 
	$$\widetilde{\Theta}:= \Theta \otimes \id_{L_\ell^2(\R)} \otimes \id_{L^2(X)}\otimes \id_{L_r^2(\R)} \colon N_0 \otm N_0^\op  \otm \B(L_\ell^2(\R) \otimes L^2(X)  \otimes L_r^2(\R)) \to \B(K).$$ 
Observe that by Lemma \ref{lemma for free quantum group1}, $\pi_H(\mathcal{M}_0)$ and $\rho_H(\mathcal{M}_0)$ are contained in $N_0 \otm N_0^\op  \otm \B(L_\ell^2(\R) \otimes L^2(X)  \otimes L_r^2(\R))$. 
Recall that for $a,b \in N$, $\pi_H(a)$ and $\theta_H(b^\op )$ are given by $\pi_{\sigma^{\varphi_N}}(a)$ on $L_\ell^2(\R) \otimes L_\ell^2(N)$ and $\theta_{\sigma^{\varphi_N}}(b^\op )$ on $L_r^2(N)\otimes L_r^2(\R)$. So if $a$ and $b$ are eigenvectors, they are of the form $\pi_H(a) = f \otimes a$ and $\theta_H(b^\op ) = b^\op  \otimes g$ for some $f,g \in C_{b}(\R)$ by Lemma \ref{lemma for free quantum group1}. 
It then holds that for any $x,y \in \R\ltimes X$,
\begin{eqnarray*}
	&&\widetilde{\Theta}(\pi_H(xa) \theta_H(b^\op  y^\op )) - \pi_K(xa) \theta_K(b^\op  y^\op ) \\
	&&\widetilde{\Theta}(\pi_H(x)\pi_H(a) \theta_H(b^\op ) \theta_H( y^\op )) - \pi_K(x)\pi_K(a) \theta_K(b^\op )\theta_K( y^\op ) \\
	&=& \widetilde{\Theta}(\pi_H(x) (a \otimes b^\op \otimes f \otimes 1_{L^2(X)} \otimes g) \theta_H( y^\op )) - \pi_K(x) (ab^\op \otimes f \otimes 1_{L^2(X)} \otimes g)\theta_K( y^\op ) \\
	&=& \pi_K(x) ((\Theta(a \otimes b^\op )-ab^\op )\otimes f \otimes 1_{L^2(X)} \otimes g) \theta_K( y^\op )).
\end{eqnarray*}
Since $\Theta(a \otimes b^\op )-ab^\op  \in \K(L^2(N))$ and $\pi_K(x),\theta_K( y^\op ) \in \C 1_N \otm \B(L_\ell^2(\R)\otimes  L^2(X)\otimes L_r^2(\R))$, the last term above is contained in $\K(L^2(N)) \otm \B(L_\ell^2(\R)\otimes L^2(X)\otimes L_r^2(\R)))$. Then by Lemma \ref{lemma for free quantum group2}, we obtain the conclusion.
\end{proof}

\begin{Lem}\label{lemma for main thm}
	Let $\Omega$ be a state on $\B(K)$ satisfying for any $x\in \widetilde{M}$ and $a\in \mathcal{U}(A)$,
	$$\Omega(\pi_K( \pi(x) ))= \Tr(pxp) \quad \text{and} \quad \Omega( \pi_K( \pi(a) \theta(\bar{a})))=1.$$
If $A \not\preceq_{\widetilde{M}} \widetilde{B}$, then using the flip $\Sigma_{12}\colon K\simeq L^2(N)\otimes L_\ell^2(\R) \otimes L^2(X)\otimes L_r^2(\R)$, it holds that
	$$\Omega\circ\Ad(\Sigma_{12})(\K(L^2(N)) \otm \B(L_\ell^2(\R)\otimes L^2(X) \otimes L_r^2(\R) ))=0.$$
\end{Lem}
\begin{proof}
	Since $\Omega$ is a state, by the Cauchy--Schwarz inequality, we have only to show that $\Omega\circ\Ad(\Sigma_{12})(\K(L^2(N)) \otm \C1_{L_\ell^2(\R)\otimes L^2(X) \otimes L_r^2(\R)} )=0$. 

	In this setting we can follow the proof of \cite[Lemma 4.4]{Is13}. Indeed suppose by contradiction that there exist $\delta>0$ and a finite subset $\mathcal{F}\subset \mathrm{Irred}(\G)$ such that
\begin{equation*}
	\Omega(1_{L_\ell^2(\R)}\otimes P_{\mathcal{F}} \otimes 1_{L^2(X)\otimes L_r^2(\R)})>\delta,
\end{equation*}
where $P_\mathcal{F}$ is the orthogonal projection onto $\sum_{x\in \mathcal{F}}H_x \otimes H_{\bar x}$. 
Then the argument in \cite[Lemma 4.4]{Is13} works by replacing $\|\, \cdot \, \|$ with $\Omega$. Hence we omit the proof.
\end{proof}

	Now we are in position to prove the main theorem. We actually prove the following more general theorem. Theorem \ref{thmA} then follows immediately with Theorem \ref{weakly compact action on core}.

\begin{Thm}\label{theorem for main thm}
	Let $A\subset p\widetilde{M}p$ be a von Neumann subalgebra and $\mathcal{G} \leq \mathcal{N}_{p\widetilde{M}p}(A)$ a subgroup. Assume the following three conditions.
\begin{itemize}
	\item[$\rm (A)$] The group $\mathcal{G}$ acts on $A$ by conjugation as a weakly compact action for $(\widetilde{M}, \pi, \theta, \mathcal{M})$.
	\item[$\rm (B)$] The quantum group $\widehat{\mathbb{G}}$ is bi-exact and centrally weakly amenable.
	\item[$\rm (C)$] We have $A\not\preceq_{\widetilde{M}}\widetilde{B}$.
\end{itemize}
Then there is a $(\mathcal{U}(A)\cup \mathcal{G})$-central state on $p\langle \widetilde{M}, \widetilde{B}\rangle p$ which coincides with $\Tr$ on $p\widetilde{M}p$. In particular the von Neumann algebra generated by $A$ and $\mathcal{G}$ is amenable relative to $\widetilde{B}$.
\end{Thm}
\begin{proof}
	By Remark \ref{rem relative weakly compact action}, we may assume $\mathcal{U}(A)\subset \mathcal{G}$. Recall from Lemma \ref{weakly contained for amenable crossed product} that as $\mathcal{M}$-bimodules,
$$ L^2(\mathcal{M}) \prec L^2(\mathcal{M}) \otimes_{(N\otimes X) } L^2(\mathcal{M}) =K ,$$
and we denote by $\nu$ the associated $\ast$-homomorphism. Let $(\xi_i)_i \subset L^2(\mathcal{M})$ be a net for the given weakly compact action of $\mathcal{G}$ and put a state $\Omega(X) := \mathrm{Lim}_i  \langle \nu(X) \xi_i, \xi_i \rangle_{L^2(\mathcal{M})}$ on $\mathrm{C}^*\{ \pi_K(\mathcal{M}), \theta_K(\mathcal{M}^\op ) \}$. 
Observe that it satisfies 
\begin{itemize}
	\item[$\rm (i)'$] $ \Omega(\pi_K( \pi(x) ))= \Tr(pxp)$ for any $x\in \widetilde{M}$;
	\item[$\rm (ii)'$] $ \Omega( \pi_K( \pi(a) \theta(\bar{a})))=1$ for any $a\in \mathcal{U}(A)$;
	\item[$\rm (iii)'$] $\Omega(\pi_K( \pi(u)\theta(\bar{u})) \theta_K(\pi(u^*)^\op \theta(u^\op )^\op )) =1$ for any $u \in \mathcal{G}$. 
\end{itemize}
Note that since $\mathcal{J}_{\mathcal{M}}\xi_i = \xi_i$, we also have $ \Omega(\theta_K( \pi(x)^\op  ))= \Tr(pxp)$ for any $x\in \widetilde{M}$. 
Write as $\nu_{K,H}$ the (not necessarily bounded) $\ast$-homomorphism for $\mathcal{M}$-bimodules $H$ and $K$. Here we claim that, using the bi-exactness of $\widehat{\G}$, the functional $\widetilde{\Omega}:=\Omega\circ \nu_{K,H}$ satisfies the following boundedness condition.
\begin{claim}
	The functional $\widetilde{\Omega}$ is bounded on $\mathrm{C}^*\{ \pi_H(\mathcal{M}_0), \theta_H(\mathcal{M}_0^\op ) \}$. 
\end{claim}
\begin{proof}[Proof of Claim]
	We first extend $\Omega$ on $\B(K)$ by the Hahn--Banach theorem. Then by Lemma \ref{lemma for main thm}, using assumption (C) and conditions $\rm (i)'$ and $\rm (ii)'$, one has
	$$\Omega\circ\Ad(\Sigma_{12})(\K(L^2(N)) \otm \B(L_\ell^2(\R)\otimes L^2(X) \otimes L_r^2(\R) ))=0.$$
Let $\Theta$ be a u.c.p.\ map for bi-exactness of $\widehat{\mathbb{G}}$ and denote by $\widetilde{\Theta}$ the extension given in Lemma \ref{lemma for free quantum group3}. Define a state on $\mathrm{C}^*\{ \pi_H(\mathcal{M}_0), \theta_H(\mathcal{M}_0^\op ) \}$ by $\widehat{\Omega}:= \Omega \circ \widetilde{\Theta}$. Then conclusions of Lemmas \ref{lemma for free quantum group3} and \ref{lemma for main thm} shows that for any $a,b \in N_0$ and $x,y \in \R\ltimes X$,
\begin{eqnarray*}
	\widehat{\Omega}(\pi_H(xa) \theta_H(b^\op  y^\op ))
	=\Omega\circ\widetilde{\Theta}(\pi_H(xa) \theta_H(b^\op  y^\op ))
	=\Omega(\pi_K(xa) \theta_K(b^\op  y^\op )).
\end{eqnarray*}
This means that the functional $\widetilde{\Omega}$ coincides with $\widehat{\Omega}$ on $\ast\text{-alg}\{ \pi_H(\mathcal{M}_0), \theta_H(\mathcal{M}_0^\op ) \}$, and hence it is a state on $\mathrm{C}^*\{ \pi_H(\mathcal{M}_0), \theta_H(\mathcal{M}_0^\op ) \}$ since so is $\widehat{\Omega}$.
\end{proof}

	We next show that the above boundedness extends partially, using the central weak amenability and a normality of $\widetilde{\Omega}$. This is the second use of the weak amenability. Recall that $\mathcal{M}$ is generated by a copy of $\widetilde{M}$ and $\widetilde{M}^\op $. We put $\widetilde{M}_0 \subset \mathcal{M}_0$ as the C$^*$-subalgebra generated by $\widetilde{B}$ and $N_0$, and note that Lemma \ref{lemma for free quantum group2} is applied to $\widetilde{M}_0$.
\begin{claim}
	The functional $\widetilde{\Omega}$ is bounded on 
	$$\mathrm{C}^*\{ \pi_H(\widetilde{M}), \pi_H(\widetilde{M}^\op ), \theta_H(\widetilde{M}^\op ), \theta_H(\widetilde{M}) \} =: \mathfrak{A},$$
where $\theta_H(\widetilde{M})$ should be understood as $\theta_H((\widetilde{M}^\op )^\op )$.
\end{claim}
\begin{proof}[Proof of Claim]
	Let $(\psi_i)_i$ be a net of finite rank normal c.c.\ maps on $N$ as in Theorem \ref{weakly amenable}. Up to convex combinations, we may assume $\psi_i\to \id_N$ in the point $\ast$-strong topology. For each $i$ we put $\psi_i^\op :=J_N\psi_i(J_N\, \cdot \, J_N)J_N$ as a normal c.c.\ map on $N^\op $. For each $i$, since $\psi_i$ commutes with the modular action, one can define a normal c.c.\ map on $\mathfrak{A}$ by
	$$\Psi_i:= \id_{L^2_\ell(\R)}\otimes\psi_i \otimes\id_{L^2(X)}\otimes\psi_i^\op \otimes \id_{L^2_r(\R)} .$$
Observe that the restriction of $\Psi_i$ on $\pi_H(\widetilde{M})$ defines a normal c.c.\ map $\widetilde{\psi}_i\colon \widetilde{M}\to \widetilde{M}_0$ (use Lemma \ref{lemma for free quantum group2}). The same holds for $\theta_H(\widetilde{M}^\op )$ and define $\widetilde{\psi}_i^\op $ similarly. 
Then with the formula $\|\pi_H(z)\|_{2,\widetilde{\Omega}} = \|zp\|_{2,\Tr} = \|\theta_H(\bar{z})\|_{2,\widetilde{\Omega}}$ for $z\in \widetilde{M}$ and by the Cauchy--Schwarz inequality, it holds that for any $a,b,x,y\in \widetilde{M}$
\begin{eqnarray*}
	&& \left|\widetilde{\Omega}\circ \Psi_i (\pi_H(ax^\op ) \theta_H(b^\op  y)) - \widetilde{\Omega} (\pi_H(ax^\op ) \theta_H(b^\op y)) \right| \\
	&=& \left|\widetilde{\Omega} (\pi_H(\widetilde{\psi}_i(a)x^\op ) \theta_H(\widetilde{\psi}_i^\op (b^\op ) y)) - \widetilde{\Omega} (\pi_H(ax^\op ) \theta_H(b^\op y)) \right| \\
	&\leq& \|\widetilde{\psi}_i(a)^*-a^*\|_{2, \Tr} \| x\|_\infty \|b\|_\infty \|y\|_\infty + \|\widetilde{\psi}_i(b)^*-b^* \|_{2, \Tr} \| a\|_\infty \|x\|_\infty\|y\|_\infty \\
	&\to& 0, \quad \text{as } i \to \infty.
\end{eqnarray*}
Hence $\widetilde{\Omega} \circ \Psi_i$ converges pointwisely to $\widetilde{\Omega}$ on the norm dense $\ast$-subalgebra $\mathfrak{A}_0\subset \mathfrak{A}$ generated by $\pi_H(\widetilde{M}), \pi_H(\widetilde{M}^\op ), \theta_H(\widetilde{M}^\op )$, and  $\theta_H(\widetilde{M})$. 
Observe that $\|\widetilde{\Omega} \circ \Psi_i|_\mathfrak{A} \| \leq 1$ for all $i$, since the range of $\Psi_i$ is contained in $\mathrm{C}^*\{ \pi_H(\mathcal{M}_0), \theta_H(\mathcal{M}_0^\op ) \}$ and $\widetilde{\Omega}$ is bounded by one on this C$^*$-algebra by the previous claim. So we conclude $\|\widetilde{\Omega}|_\mathfrak{A}\|\leq 1$, as desired.
\end{proof}

	Observe that $\widetilde{\Omega}$ is a state, since it is positive on $\mathfrak{A}_0$ by construction, and $\widetilde{\Omega}(1)=1$. 
By the Hahn--Banach theorem, we extend $\widetilde{\Omega}$ from $\mathfrak{A}$ to $\B(H)$ which we still denote by $\widetilde{\Omega}$. By construction, it satisfies that for all $x\in \widetilde{M}$ and $u\in \mathcal{G}$,
	$$\widetilde{\Omega}(\pi_H(x)) = \Tr(pxp) \quad \text{and} \quad \widetilde{\Omega}(\pi_H( \pi(u)\theta(\bar{u})) \theta_H(\pi(u^*)^\op \theta(u^\op )^\op )) =1.$$
Putting $U(u):=\pi_H( \pi(u)\theta(\bar{u})) \theta_H(\pi(u^*)^\op \theta(u^\op )^\op )$,  the second condition implies $\widetilde{\Omega} (Y)=\widetilde{\Omega} (U(u)Y U(u)^*)$ for any $u \in \mathcal{G}$ and $Y\in \B(H)$. 
Recall that since $H=L^2(\mathcal{M})\otimes_X L^2(\mathcal{M})$, regarding $L^2(\mathcal{M})$ as a $\langle\mathcal{M}, \R \ltimes X \rangle$-$X$-bimodule, the basic construction $\langle\mathcal{M}, \R \ltimes X \rangle$ acts on $H$ by left, which we again denote by $\pi_H$, whose image commutes with $\theta_H(\mathcal{M}^\op )$. 
So if $Y \in \langle\mathcal{M}, \R \ltimes X \rangle \cap \theta(\widetilde{M}^\op )'$, then
	$$\widetilde{\Omega} (\pi_H(Y)) = \widetilde{\Omega} (U(u)\pi_H(Y) U(u)^*)=\widetilde{\Omega} (\pi_H(\pi(u))\pi_H(Y) \pi_H(\pi(u))^*)$$ 
for any $u \in \mathcal{G}$. So the state $\widetilde{\Omega}\circ\pi_H$ is a $\mathcal{G}$-central state on $\langle\mathcal{M}, \R \ltimes X \rangle \cap \theta(\widetilde{M}^\op )'$. 
Finally since $\widetilde{M} L^2(\R \ltimes X) \subset L^2(\mathcal{M})$ is dense, the von Neumann subalgebra in $\langle\mathcal{M}, \R \ltimes X \rangle \cap \theta(\widetilde{M}^\op )'$ generated by $\widetilde{M}$ and $e_{\R \ltimes X}:= 1_{L^2(\R)}\otimes P_N \otimes 1_X$, where $P_N$ is the 1-dim projection onto $\C\Lambda_{\varphi_N}(1_N)$, is canonically identified as $\langle \widetilde{M}, \widetilde{B} \rangle$ (by the fact that $e_{\R \ltimes X}\, a \, e_{\R \ltimes X} = E_{\widetilde{B}}(a) e_{\R \ltimes X}$ for $a\in \widetilde{M}$). 
Thus the restriction of $\widetilde{\Omega}\circ \pi_H$ on $\langle \widetilde{M}, \widetilde{B} \rangle$ is a $\mathcal{G}$-central state, which coincides with $\Tr$ on $p\widetilde{M}p$. 
Using the normality on $p\widetilde{M}p$ and by the Cauchy--Schwarz inequality, we obtain that $\mathcal{G}''$ is amenable relative to $\widetilde{B}$ in $\widetilde{M}$.
\end{proof}

\subsection{\bf Proof of Corollary \ref{corB}}\label{Proof of Corollary B}

\begin{proof}[Proof of Corollary \ref{corB}]
	Put $M:=N\ovt B \supset N_0 \ovt B=: M_0$ and suppose that $A\subset  M_0$ is a Cartan subalgebra. We will deduce a contradiction. For this, let $R_\infty$ be the AFD III$_1$ factor and $A_0\subset R_\infty$ a Cartan subalgebra. Up to exchanging $B$ and $A$ with $B\ovt R_\infty$ and $A\ovt A_0$ respectively, we assume that $B$ is a type III$_1$ factor (e.g.\ Lemma \ref{center of core of III1}).

	Let $\psi_{N_0}$ and $\tau_A$ be faithful normal states on $N_0$ and $A$ respectively, and $E_{N_0}$ and $E_A$ faithful normal conditional expectations from $N$ to $N_0$ and from $M_0$ to $A$ respectively. Put $\psi_A:=\tau_A\circ E_{A}$, $\psi_{N}:=\psi_{N_0}\circ E_{N_0}$, $\psi:=\psi_N\otimes \varphi_B$, $\varphi:=h\otimes\varphi_B$ and $E_{M_0}:=E_{N_0}\otimes \id_B$. Then since all continuous cores are isomorphic, we have $\Pi_{\psi_A\circ E_{M_0},\psi}\colon C_{\psi}(M) \to C_{\psi_A \circ E_{M_0}}(M)$, which restricts to $\Pi_{\psi_A, \psi_{N_0}\otimes \varphi_B}\colon C_{\psi_{N_0}\otimes \varphi_B}(M_0)\to C_{\psi_A}(M_0)$. Recall that $A\ovt L\R \subset C_{\psi_A}(M_0)$ is a Cartan subalgebra (e.g.\ \cite[Proposition 2.6]{HR10}) and hence so is the image $\widetilde{A}:=\Pi_{\varphi, \psi_A \circ E_{N_0}}(A\ovt L\R)\subset \Pi_{\varphi, \psi_A \circ E_{N_0}}(C_{\psi_A}(M_0))=:\mathcal{N}$.

\begin{claim}
	There is a conditional expectation $E \colon \langle C_{\varphi}(M), C_{\varphi_B}(B) \rangle \to \mathcal{N}$ which is faithful and normal on $C_\varphi(M)$.
\end{claim}
\begin{proof}
	We first show $A\not\preceq_{M} B$. Indeed, if $A\preceq_{M} B$, then we have $A\preceq_{M_0} B$ by Lemma \ref{intertwining lemma}. So by \cite[Lemma 4.9]{HI15}, one has $N_0=B' \cap M_0 \preceq_{M_0} A'\cap M_0=A$, which is a contradiction. Hence we have $A\not\preceq_{M} B$. 

	We apply \cite[Proposition 2.10]{BHR12} (this holds if $A$ is finite by exactly the same proof) and get $\widetilde{A} \not\preceq_{C_{{\varphi}}({M})} C_{\varphi_{{B}}}({B})$. Fix any projection $p\in \widetilde{A}$ with $\Tr(p)<\infty$, where $\Tr$ is the canonical trace on the core, and observe $p\widetilde{A}p \not\preceq_{C_{{\varphi}}({M})} C_{\varphi_{{B}}}({B})$ by definition. 
We apply Theorem \ref{thmA} to $p\widetilde{A}p$ and get that $\mathcal{N}_{pC_\varphi(M)p}(p\widetilde{A}p)''$ is amenable relative to $C_{\varphi_B}(B)$. 
Observe that $\mathcal{N}_{pC_\varphi(M)p}(p\widetilde{A}p)'' = p(\mathcal{N}_{C_\varphi(M)}(\widetilde{A})'')p$ (e.g\ \cite[Proposition 2.7]{HR10}). 
Combined with \cite[Remark 3.3]{Is17}, there is a conditional expectation $E_p \colon p\langle C_{\varphi}(M), C_{\varphi_B}(B) \rangle p \to  p\mathcal{N} p$ which restricts to the $\Tr$-preserving expectation on $pC_{\varphi}(M)p$. Taking a net $(p_i)_i$ of $\Tr$-finite projections converging to 1 weakly, one can construct a desired conditional expectation by $E(x):= \sigma\text{-weak}\mathrm{Lim}_i E_{p_i}(p_ixp_i)$ for $x \in \langle C_{\varphi}(M), C_{\varphi_B}(B) \rangle$.
\end{proof}

We apply \cite[Theorem 3.2]{Is17} to the conclusion of the claim and get that $M_0$ is amenable relative to $B$ in $M$. Hence there is a conditional expectation $F \colon \langle M,B \rangle \to M_0$ which is faithful and normal on $M$. Using the identification $\langle M,B \rangle= \B(L^2(M))\ovt B$, we can construct a conditional expectation from $\B(L^2(M))$ onto $N_0$, that means $N_0$ is injective. This is a contradiction.
\end{proof}

\small{

}
\end{document}